\documentclass[12pt]{amsart}
\usepackage{amscd,amsmath,amsthm,amssymb}
\usepackage[left]{lineno}
\usepackage{color}
\usepackage{stmaryrd}
\usepackage[utf8]{inputenc}
\usepackage{cleveref}
\usepackage{epstopdf}
\usepackage{graphicx}
\usepackage{xcolor}

\usepackage{tikz}
\definecolor{verylight}{gray}{0.97}
\definecolor{light}{gray}{0.9}
\definecolor{medium}{gray}{0.85}
\definecolor{dark}{gray}{0.6}
\newcommand\calP{\mathcal{P}}
\def\NZQ{\mathbb}               % the font for N,Z,Q,R,C

\def\ZZ{{\NZQ Z}}

\def\KK{{\NZQ K}}

\def\KK{{\NZQ K}}

\def\G{{\mathcal G}}

\def\C{{\mathcal C}}

\def\ab{{\mathbf a}}
\def\bb{{\mathbf b}}
\def\xb{{\mathbf x}}

\def\0b{{\mathbf 0}}

\def\reg{{\mathbf reg}}
\def\height{\operatorname{ht}}
\def\depth{\operatorname{depth}}
\def\opn#1#2{\def#1{\operatorname{#2}}} % to make operators
\opn\chara{char} \opn\length{\ell} \opn\pd{pd} \opn\rk{rk}
\opn\projdim{proj\,dim} \opn\injdim{inj\,dim} \opn\rank{rank}
\opn\depth{depth} \opn\grade{grade} \opn\height{height}
\opn\embdim{emb\,dim} \opn\codim{codim}

\opn\Tr{Tr} \opn\bigrank{big\,rank}
\opn\superheight{superheight}\opn\lcm{lcm}
\opn\trdeg{tr\,deg}%\emph{
	\opn\reg{reg} \opn\lreg{lreg} \opn\ini{in} \opn\lpd{lpd}
	\opn\size{size} \opn\sdepth{sdepth}
	\opn\link{link}\opn\fdepth{fdepth}\opn\lex{lex}
	\opn\tr{tr}
   \opn\d{d}

	\opn\type{type}
	\opn\gap{gap}
	\opn\arithdeg{arith-deg}
	\opn\HS{HS}
	\opn\GL{GL}
	\opn\div{div} \opn\Div{Div} \opn\cl{cl} \opn\Cl{Cl}
	\opn\Spec{Spec} \opn\Supp{Supp} \opn\supp{supp} \opn\Sing{Sing}
	\opn\Ass{Ass} \opn\Min{Min}\opn\Mon{Mon}
	\opn\Ann{Ann} \opn\Rad{Rad} \opn\Soc{Soc}\opn\Deg{Deg}
	\opn\Im{Im} \opn\Ker{Ker} \opn\Coker{Coker} \opn\Am{Am}
	\opn\Hom{Hom} \opn\Tor{Tor} \opn\Ext{Ext} \opn\End{End}
	\opn\Aut{Aut} \opn\id{id}
	
	\opn\nat{nat}
	\opn\pff{pf}%   \pf exists already
	\opn\Pf{Pf} \opn\GL{GL} \opn\SL{SL} \opn\mod{mod} \opn\ord{ord}
	\opn\Gin{Gin} \opn\Hilb{Hilb}\opn\sort{sort}
	\opn\PF{PF}\opn\Ap{Ap}
	\opn\mult{mult}
	\opn\bight{bight}
	\opn\aff{aff}
	\opn\relint{relint} \opn\st{st}
	\opn\lk{lk} \opn\cn{cn} \opn\core{core} \opn\vol{vol}  \opn\inp{inp} \opn\nilpot{nilpot}
	\opn\link{link} \opn\star{star}\opn\lex{lex}\opn\set{set}
	\opn\width{wd}
	\opn\Fr{F}
	\opn\QF{QF}
	\opn\G{G}
	\opn\type{type}\opn\res{res}
	\opn\conv{conv}
	\opn\Ind{Ind}
	\opn\gr{gr}
	
	\def\pot#1#2{#1[\kern-0.28ex[#2]\kern-0.28ex]}

	\opn\dirlim{\underrightarrow{\lim}}
	\opn\inivlim{\underleftarrow{\lim}}
	\let\to=\rightarrow
	
	\def\Implies{\ifmmode\Longrightarrow \else
		\unskip${}\Longrightarrow{}$\ignorespaces\fi}
	\def\implies{\ifmmode\Rightarrow \else
		\unskip${}\Rightarrow{}$\ignorespaces\fi}
	\def\iff{\ifmmode\Longleftrightarrow \else
		\unskip${}\Longleftrightarrow{}$\ignorespaces\fi}

	\let\:=\colon
	\newtheorem{Theorem}{Theorem}[section]
	\newtheorem{Lemma}[Theorem]{Lemma}
	\newtheorem{Corollary}[Theorem]{Corollary}
	
	\newtheorem{Remark}[Theorem]{Remark}

	\newtheorem{Definition}[Theorem]{Definition}
	
	\newtheorem{Conjecture}[Theorem]{Conjecture}

	\newtheorem{Setting}[Theorem]{Setting}
	\newtheorem{Claim}[Theorem]{Claim}
	
	\let\epsilon\varepsilon
	\let\kappa=\varkappa
	\def\qed{\ifhmode\textqed\fi
		\ifmmode\ifinner\quad\qedsymbol\else\dispqed\fi\fi}
	\def\textqed{\unskip\nobreak\penalty50
		\hskip2em\hbox{}\nobreak\hfil\qedsymbol
		\parfillskip=0pt \finalhyphendemerits=0}
	\def\dispqed{\rlap{\qquad\qedsymbol}}
	
	\opn\dis{dis}
	\def\pnt{{\raise0.5mm\hbox{\large\bf.}}}
	
	\opn\Lex{Lex}
	
\begin{document}
		%\linenumbers
		\title {The regularity of  monomial ideals and their integral closures}
		
		\author {Yijun Cui, Cheng Gong and Guangjun Zhu}

		\address{School of Mathematical Sciences, Soochow University, Suzhou, Jiangsu, 215006, P. R. China}

		\email{237546805@qq.com, cgong@suda.edu.cn, %zmtang@suda.edu.cn, \linebreak[4]
zhuguangjun@suda.edu.cn}
		%	\linebreak[4]
		
		%\thanks{$^{\ast}$ Corresponding author}
		%	\linebreak[4]
		
		% \thanks{$^{\ast}$ Corresponding author}
		


		% \subjclass[2010]{Primary 13C99; Secondary 13E15, 13A15.}
		%		13H10   	Special types (Cohen-Macaulay, Gorenstein, Buchsbaum, etc.)
		%		13D02   	Syzygies, resolutions, complexes
		%		05E40   	Combinatorial aspects of commutative algebra
		%		16S36   	Ordinary and skew polynomial rings and semigroup rings
		
		%		14M25   	Toric varieties, Newton polyhedra [See also 52B20]
		%		13A02   	Graded rings
		%		13F20   	Polynomial rings and ideals; rings of integer-valued polynomials
		%		13A18   	Valuations and their generalizations
		%		06A11   	Algebraic aspects of posets
		%       05C38       Paths and cycles
		%       13A15       Ideals; multiplicative ideal theory
		%       13F20       Polynomial rings and ideals; rings of integer-valued
		
		%\keywords{   }
		
		\maketitle
		\begin{abstract}
		Let $I$ be a monomial ideal  in a polynomial ring $S=\KK[x_1,\ldots,x_n]$  over a field $\KK$ with $n=2$ or $3$, and let $\overline{I}$ be its integral closure. We will show that
 $\reg (\overline{I}) \le \reg (I)$.  Furthermore, if $I$ is generated by  elements of degree $d$,  then  $\reg (I)=d$ if and only if
 $I$ has linear quotients.
		\end{abstract}
		\setcounter{tocdepth}{1}
		%\tableofcontents

		%	\section*{Introduction}
		\section{Introduction}
	In \cite{KP}, K\"uronya and Pintye  studied the relationship between the log--canonical thresholds and the Castelnouvo-Mumford regularity (or simply regularity)  of the coherent sheaves of ideals in projective spaces. As a byproduct, they came up with the following striking  conjecture:

\begin{Conjecture}\label{conj_KP} Let $S=\KK[x_1,\ldots,x_n]$ be  a polynomial ring with  $n$ variables over a field $\KK$ and let $I\subset S$ be
 a homogeneous ideal.  Then
\[
\reg (\overline{I}) \le \reg (I),
\]
where   $\overline{I}$ is the integral closure of $I$, and $\reg (I)$ is  the   regularity of $I$.
\end{Conjecture}

In this paper, we  investigate whether this conjecture holds for monomial
ideals. Since little is known about this conjecture, even for monomial ideals. This is mainly due to the difficulty of computing integral
closure and regularity of an ideal.

For certain special monomial ideals, such as the edge ideal $I(G)$ of a simple graph $G$, Nguyen and Vu \cite{MV} showed that $\reg(\overline{I(G)^s})\leq \reg(I(G)^s)$ for all $s\le 4$.  If $G$ has at most one odd cycle — for example, if $G$ is a bipartite or unicyclic graph — then its edge ideal, $I(G)$, is normal. Thus,
$\reg (\overline{I(G)^s})=\reg (I(G)^s)$ for all $s\ge 1$. In  \cite{KK}, Kumar and Kumar  showed that  $\reg (\overline{I(G)})=\reg (I(G))$ if  $G$ is an odd bicyclic graph or a bow graph. In \cite{NVZ}, Nguyen, Vu and the third author of this paper demonstrated
that this conjecture holds for the edge ideals of arbitrary weighted oriented graphs.

In this paper,  we will demonstrate that the conjecture holds for  monomial
ideals if $n=2$ or $3$. We will also demonstrate that when $n=2$ or $3$, for an equigenerated monomial ideal of degree
 $d$, $\reg(I)=d$  if and only if $I$ has linear quotients.

The article is organized as follows: Section \ref{sec:prelim} provides the definitions and basic facts that will be used throughout the paper. Section $3$ gives  a characterization of the regularity of an  equigenerated monomial ideal of degree $d$ is $d$, and uses this characterization to prove that  Conjecture \ref{conj_KP} holds.

\section{Preliminaries}
		\label{sec:prelim}
In this section, we will provide the definitions and basic facts that  will be used throughout this paper.  For more detailed information, please refer to \cite{E} and \cite{HH}.

Let $S=\KK[x_1, \ldots, x_n]$ be a polynomial ring with  $n$ variables over a field $\KK$. Let $I\subset S$ be a monomial ideal and let $\mathcal{G}(I)$ be the unique minimal set of its monomial generators. We say that $I$ has \textit{linear quotients}, if there exists an ordering $u_1,\ldots,u_m$ of $\mathcal{G}(I)$ such that for each $i=1,2,\ldots,m-1$, the colon ideal $(u_1,...,u_{i}): u_{i+1}$ is
generated by a subset of the variables. In this paper, we will  use the following famous conclusion many times: If $I$ has linear quotients, then $I$ has a linear resolution (see  \cite[Lemma 4.1]{CH}).

An element $f\in S$ is said to be \textit{integral} over $I$ if
 it satisfies an integral dependence relation
 \[
f^{\ell}+c_1f^{\ell-1}+\dots+c_{\ell-1}f+c_{\ell}=0 \text{\ \ with\ \ }c_i \in I^i \text{\ for any\ } i\in [\ell].
 \]
 The set of elements $u$ in $S$ that are integral over $I$ forms an ideal of $S$,
 called the \textit{ integral closure} of $I$, denoted by $\overline{I}$. The ideal $I$ is \textit{ integrally closed} if $\overline{I}=I$.
 If $I$ is a monomial ideal, then  $\overline{I}$ is  generated by all monomials $u$ such that  $u^{k}\in I^k$ for some $k\ge 1$, see, e.g., \cite[Theorem 1.4.2]{HH}.
Moreover, in this case, it can be described geometrically.
Let $\mathcal{G}(I)=\{\xb^{{\bb}_1},\ldots, \xb^{{\bb}_m}\}$, where ${\bb}_i=(b_{i1},\ldots,b_{in})$ is the exponent vector of $\xb^{{\bb}_i}$ for each $i\in [m]$ and  $[m]=\{1,\ldots,m\}$, then
$\overline{I}$  is a monomial ideal defined by the integral convex hull of the exponent vectors of $I$. That is,
$\overline{I}$ is  generated by the set $\{\xb^{\alpha}\mid \alpha\in  \mathcal{C}(I)\cap \mathbb{Z}^n\}$ of monomials, where $\mathcal{C}(I)$ is the  \textit{Newton polyhedron} of $I$ and is defined as
 \begin{align*}
    \mathcal{C}(I)&=conv(\mathcal{G}(I))+\mathbb{R}^n_+\\
    and  \ \ \ \ \ \ \ \  \ \ \ \ \ \ \ \ \  \ \ \ \ \ \ \ \ \ \ \  \ \ \ & \ \\
     conv(\mathcal{G}(I))&=conv({\bb}_1,\dots,{\bb}_m)\\
    &=\{\ab=\sum\limits_{i=1}\limits^{m}{\lambda_i{\bb}_i} \mid \sum\limits_{i=1}\limits^{m}{\lambda_i}=1, \lambda_i \in \mathbb{Q}_{+}\}.
   \end{align*}
Note that $\mathcal{C}(I)$ is a pointed polyhedron; that is to say, $\mathcal{C}(I)$ contains no trivial linear subspaces of $\mathbb{R}^n$.
 Therefore, each minimal face of $\mathcal{C}(I)$ consists of just one point, known as a vertex of $\mathcal{C}(I)$. The set of vertices of $\mathcal{C}(I)$, denoted $V(I)$, is uniquely defined by $\mathcal{C}(I)$ and is a subset of $\mathcal{G}(I)$.

\begin{Lemma}{\em (\cite[Proposition 12.1.4]{V})}\label{e0}
   Let $I\subset S$ be a monomial ideal with $\mathcal{G}(I)=\{\xb^{{\bb}_1},\ldots, \xb^{{\bb}_m}\}$. Then $\overline{I}$ is generated by the
monomials  $\xb^{\ab}$, where $\ab=(\lceil a_1 \rceil,\dots,\lceil a_n \rceil)$ with $(a_1,\ldots,a_n)\in \text{conv}(\mathcal{G}(I))$ and   each $\lceil a_i \rceil$ is the smallest integer $\ge a_i$.
\end{Lemma}

The regularity of graded modules over the polynomial ring $S$ can be defined in various ways.
Let $\frak{m}$ be the  unique graded maximal  ideal in $S$. The local cohomology modules of a finitely generated
graded $S$-module $M$ with respect to $\frak{m}$  are denoted by $H_{\frak{m}}^i(M)$
for $i\in \ZZ$.
\begin{Definition}
    Let $M$ be a finitely generated graded $S$-module.
    For $i\ge 0$, let
    \[
    a_i(M)=\begin{cases}
			\max\{t\in \ZZ: [H_{\frak{m}}^i(M)]_t \ne 0\},& \text{if\ } H_{\frak{m}}^i(M)\ne 0,\\
-\infty,& \text{otherwise.}
				\end{cases}
    \]
   The {\em regularity} of $M$ is defined to be
\[
 \reg(M)= \max\{a_i(M) + i:i\ge 0\}.
 \]
\end{Definition}
Since $a_i(M)=0$ for $i>\dim(M)$, the regularity of $M$ is well defined. This invariant can also be computed via the minimal free resolution
\begin{Definition}
Let $M$ be a finitely generated graded $S$-module and let
$$0\rightarrow \bigoplus\limits_{j\in \ZZ}S(-j)^{\beta_{p,j}(M)}\rightarrow \cdots\rightarrow \bigoplus\limits_{j\in \ZZ}S(-j)^{\beta_{0,j}(M)}\rightarrow M\rightarrow 0$$
be its minimal free resolution. Then the regularity of $M$ is given by
\[
\reg(M)=\max\{j-i\mid \beta_{i,j}(M)\neq 0\}.
\]
\end{Definition}
Looking at the minimal free resolution, it is easy to see that $\reg(S/I)=\reg(I)-1$, so we shall work with $\reg(I)$ and $\reg(S/I)$ interchangeably.

For any vector $\mathbf{v}=(v_1,\ldots, v_n)\in \mathbb{Z}_{\ge 0}^n$, we set $|\mathbf{v}|=v_1+\cdots+v_n$. The following lemmas are often used for computing  regularity  of a module.
\begin{Lemma}  {\em (\cite[Theorem 2.7]{Hoa})}
	\label{compare}
Let $I\subseteq S$ be a monomial ideal and $\delta(I)=\max\{|\mathbf{v}| |\mathbf{v}\in V(I)\}$. Then
    \begin{align*}
        \delta(I) &\leq \d(\overline{I}) \leq \delta(I)+ \dim S/I, \\
        \delta(I) &\leq \reg(\overline{I}) \leq \delta(I) + \dim S/I,
    \end{align*}
where $\d(\overline{I})=\max\{|\mathbf{v}|\mid \xb^{\mathbf{v}}\in \mathcal{G}(\overline{I})\}$ is the maximal generating degree  of $\overline{I}$.
\end{Lemma}

\begin{Lemma}  {\em (\cite[Lemma 3.1]{HT})}
	\label{exact}
Let $0\longrightarrow M\longrightarrow N\longrightarrow P\longrightarrow 0$ be a short exact
	sequence of finitely generated graded S-modules. Then we have
	\begin{itemize}
		\item[(1)] $\reg(N)\leq max\{\reg(M), \reg(P)\}$, the equality holds if $\reg(P) \neq \reg(M)-1$.
\item[(2)] $\reg(P)\leq max\{\reg(M)-1, \reg(N)\}$, the  equality holds if $\reg(M) \neq \reg(N)$.
		\end{itemize}
\end{Lemma}

\begin{Lemma}{\em (\cite[Lemma 3.2]{HT})}
\label{sum1}
Let $S_{1}=\KK[x_{1},\dots,x_{m}]$ and $S_{2}=\KK[x_{m+1},\dots,x_{n}]$ be two polynomial rings  over a field $\KK$. Let  $I\subset S_{1}$,
$J\subset S_{2}$ be two nonzero homogeneous  ideals.  Then
\begin{itemize}
\item[(1)] $\reg(I+J)=\reg(I)+\reg(J)-1$,
\item[(2)]$\reg(IJ)=\reg(I)+\reg(J)$.
\end{itemize}
In particular,  if  $u$ is a monomial of degree  $t$,  then $\reg(uI)=t+\reg(I)$.
\end{Lemma}

 \begin{Lemma}\label{regular}{\em (\cite[Lemma 4.4]{BHT})}
 	Let  $I=(u_1,\ldots, u_m)$ be an equigenerated monomial ideals of degree $d$. If $u_1,\ldots, u_m$ is a regular sequence, then
 \[
 \reg(I^t)=dt+(d-1)(m-1)  \text{\ for all\ } t\ge 1.
 \]
 \end{Lemma}

Polarization and Betti splitting decomposition techniques are important methods and techniques for  computing regularity of a module.
\begin{Definition}{\em(\cite[Definition 2.1]{F})}\label{polarization}
Let $I\subset S$ be a monomial ideal with $\mathcal{G}(I)=\{u_1,\ldots,u_m\}$, where $u_i=\prod\limits_{j=1}^n x_j^{a_{ij}}$ for any $i\in [m]$.
The polarization of $I$, denoted by $I^{\calP}$, is the following squarefree monomial ideal in the polynomial ring $S^{\calP}$
$$I^{\calP}=(\calP(u_1),\ldots,\calP(u_m)),$$
where $\calP(u_i)=\prod\limits_{j=1}^n \prod\limits_{k=1}^{a_{ij}} x_{jk}$ is a squarefree monomial  in $S^{\calP}=\KK[x_{j1},\ldots,x_{ja_j}\mid j\in [n]]$ and $a_j=\max\{a_{ij}|i\in [m]\}$ for  any $j\in [n]$.
\end{Definition}

 The following is a  useful property of the polarization.

\begin{Lemma}{\em (see, e.g.,\cite[Corollary 1.6.3]{HH})}\label{polar}
 Let $I\subset S$ be a monomial ideal and $I^{\calP}\subset S^{\calP}$ be its polarization.
Then
\[
\beta_{ij}(I)=\beta_{ij}(I^{\calP})
\]
 for all $i$ and $j$. In particular, $\reg(I)=\reg(I^{\calP})$.
\end{Lemma}

\begin{Definition}{\em (\cite[Definition 1.1]{FHT})}\label{bettispliting}
		Let $I$ be a monomial ideal. If there exist monomial ideals $J$ and $K$ such that $\mathcal{G}(I)=\mathcal{G}(J)\cup\mathcal{G}(K)$ and $\mathcal{G}(J)\cap\mathcal{G}(K)=\emptyset$. Then $I=J+K$ is a Betti splitting if
\[
		\beta_{i, j}(I)=\beta_{i, j}(J)+\beta_{i, j}(K)+\beta_{i-1, j}(J \cap K) \text { for all } i, j \geq 0,
\]
		where $\beta_{i-1, j}(J \cap K)=0$ for $i=0$.
\end{Definition}

\begin{Lemma}{\em (\cite[Corollary 2.7]{FHT})}\label{spliting}
Suppose  $I=J+K$, where $\mathcal{G}(J)$ consists
of all the generators of $I$ that are divisible by some variable $x_i$, and $\mathcal{G}(K)$ is a nonempty set containing the remaining generators of $I$. If $J$ has a linear resolution, then $I=J+K$ is a Betti splitting.

From Definition \ref{bettispliting}, we know that if $I=J+K$ is a Betti splitting then   $\reg(I)=\max\{\reg(J), \reg(K), \reg(J \cap K)-1\}$.
\end{Lemma}

Recall that a  hypergraph  $H = (V, \mathcal{E})$ consists of a vertex set $V$ and a edge set $\mathcal{E}$ (nonempty subsets of $V$). $H$ is simple if no edges are properly contained in others.
For a subset $Y\subseteq V$, the \textit{induced subhypergraph} $H[Y]$ has vertex set $Y$ and edges $\{E \in \mathcal{E} \mid E \subseteq Y\}$.

\begin{Lemma}{\em (\cite[Lemma 3.1]{H})}\label{hyper}
 Let H be a simple hypergraph. Then $\reg(H)\geq\reg(H')$ for any induced subhypergraph $H'$ of $H$.
	\end{Lemma}

Given a monomial ideal  $I\subset S$,  for the rest of the paper we will regard its polarisation $I^{\calP}$ as the edge ideal of a hypergraph $H$.

\section{Regularity of monomial ideals and their integral closures}
\label{sec:reg}

In this section, we will establish a relationship between the regularity of the  monomial ideal and that of its integral closure.

We first show that it suffices to prove the conjecture for ideals generated in one degree, see Claim~\ref{claim:1} and Corollary~\ref{cor:2} below.

Let $I \subset S := k[x_1, \ldots, x_n]$ be a homogeneous ideal.

\begin{Claim}\label{claim:1}
Let $d := d(I)$ the maximal generating degree of $I$. Denote by $I_a$ the $d$-th component of $I$ and $I_{\geq a}$ the ideal whose homogeneous elements have degrees at least $a$. Then
\begin{enumerate}
\item For all $a \geq d$, $I_{\geq a} = (I_a)$.
\item For all $a \geq 1$, $\reg(I) \leq \reg(I_{\geq a})$ and $\reg(I) = \reg((I_d))$.
\item For all $a \geq d$, $\overline{(I_a)} = \overline{I}_{\geq a}$.
\end{enumerate}
\end{Claim}

\begin{proof}
(1) Let $\mathcal{G}(I)$ be a minimal set of generators of $I$, and $g \in \mathcal{G}(I)$. Then $\deg(g) \leq d$. If $m$ is a monomial such that $\deg(gm) \geq a \geq \deg(g)$, one can write $m = m'm''$ such that $\deg(gm') = a$. Hence $gm \in (I_a)$. Any homogeneous element $f$ of $I_{\geq a}$ has a form $\sum\limits_{g_i \in \mathcal{G}(I)} g_i m_i$ for some monomials $m_i$ such that $\deg(g_i m_i) = \deg(f) \geq a$. From the above discussion we see that $f \in (I_a)$.

(2) Note that for all $a \geq 1$, $I/I_{\geq a}$ is of finite length. Hence, from the exact sequence
\[
0 \to I/I_{\geq a} \to R/I_{\geq a} \to R/I \to 0,
\]
we get $\reg(I) \leq \reg(I_{\geq a})$. Putting $a = d$ in the above exact sequence
\[
0 \to I/I_{\geq d} \to R/I_{\geq d} \to R/I \to 0,
\]
and note that $\reg(R/I) \geq d-1$, we get
\[
\reg(I_{\geq d}) = 1 + \reg(R/I_{\geq d}) \leq 1 + \max\{d-1, \reg(R/I)\} = 1 + \reg(R/I) = \reg(I).
\]
Combining with the above inequality, we get $\reg(I) = \reg(I_{\geq d}) = \reg((I_d))$ (by (1)).

(3) It suffices to show $\supseteq$. Let $f \in \overline{I}_{\geq a}$ be a homogeneous polynomial. Then $p := \deg(f) \geq a$, and $f$ satisfies an equation of type
\[
f^m + f^{m-1}c_1 + \cdots + c_{m-1}f + c_m = 0,
\]
where $c_i \in I^i$ are homogeneous polynomials of degree $pi$. Fix $i \geq 1$. Then
\[
c_i = \sum_{g_1, \ldots, g_i \in \mathcal{G}(I)} g_1 \cdots g_i \alpha_{g_1, \ldots, g_i} m_{g_1, \ldots, g_i},
\]
where $\alpha_{g_1, \ldots, g_i} \in k$ and $\deg(m_{g_1, \ldots, g_i}) + \deg(g_1) + \cdots + \deg(g_i) = pi$. One can find monomials $m_1, \ldots, m_i, m$ such that $m_{g_1, \ldots, g_i} = m_1 \cdots m_i m$ and $\deg(m_i) + \deg(g_i) = a$. Then
\[
g_1 \cdots g_i m_{g_1, \ldots, g_i} = (g_1 m_1) \cdots (g_i m_i) m \in (I_a)^i.
\]
This implies $f \in \overline{(I_a)}$.
\end{proof}

As consequence, we get

\begin{Corollary}\label{cor:2}
If conjecture holds for equigenerated ideals, then it hold in general.
\end{Corollary}

\begin{proof}
Let $d = d(I)$. We have
\begin{align*}
\reg(I) &\leq \reg(I_{\geq d}) \text{ (by Claim~\ref{claim:1}(2))}\\
 &= \reg((I_d)) \text{ (by Claim~\ref{claim:1}(3))}\\
  &\leq \reg(\overline{(I_d)}) \text{ (by Assumption)}\\
   &= \reg(\overline{I}) \text{ (by Claim~\ref{claim:1}(2))}.
\end{align*}
\end{proof}

For a monomial $\xb^{\mathbf{a}}= x_1^{a_1} \cdots x_n^{a_n}$ in $S$ with exponent vector $\mathbf{a}=(a_1,\ldots,a_n)$, we set $\deg_{x_i}(\xb^{\mathbf{a}})=a_i$ and $\deg(\xb^{\mathbf{a}})=\sum\limits_{i=1}^n\deg_{x_i}(\xb^{\mathbf{a}})$.

\begin{Remark}\label{rem:3}
Let $f \in S$ be a homogeneous polynomial. Then it is easy to show that
\begin{enumerate}
\item $\reg(fI) = \deg(f) + \reg(I)$.
\item $\overline{fI} = f\overline{I}$, whence $\reg(\overline{fI}) = \deg(f) + \reg(\overline{I})$.
\item Let $I = (x^d, x^{d-a_1}y^{a_1}, \ldots, x^{d-a_{k-1}}y^{a_{k-1}}, y^d)$, where $a_0 := 0 < a_1 < \cdots < a_{k-1} < d =: a_k$. Since $x^{d-(a_i+1)}y^{a_{i+1}-1} \notin I$ if $a_{i+1} \geq a_i + 2$,
\[
\reg(I) \ge \max\{d + a_{i+1} - a_i - 1 \mid i = 0, \ldots, k-1\}.
\]
In particular, $\reg(I) = d$ if and only if $a_{i+1} = a_i + 1$ for all $i \leq k-1$.
\end{enumerate}
\end{Remark}

By Corollary \ref{cor:2} and Remark \ref{rem:3}(3), we have the following result:
\begin{Theorem}\label{2-variable}
Let $I\subset S=\KK[x_1,x_2]$ be a non-zero monomial ideal. Then
\[
\reg(\overline{I})\leq\reg(I).
\]
\end{Theorem}

Now we are ready to prove the first main result of this section.
\begin{Theorem}\label{e9}
Let $I\subset S=\KK[x_1,x_2,x_3]$ be an equigenerated monomial ideal.  Then $\reg(\overline{I})\leq \reg(I)$.
\end{Theorem}

To prove this theorem, we need the following series of lemmas and theorems.

By polarization and an easy indication of hypergraph and an induced hypergraph, using Lemma \ref{hyper}, we have:

\begin{Claim}\label{claim:4}
Given a monomial ideal $I$ and $a_1, \ldots, a_n \in \mathbb{N}$. Denote by $I_{x_i \leq a_i \mid i=1,\ldots,n}$ the subideal of $I$ generated by monomial generators $\mathbf{x}^{\mathbf{b}}$ of $I$ such that $b_i \leq a_i$ for all $i = 1, \ldots, n$ (if $\emptyset \neq U \subset [n]$, then $I_{x_i \leq a_i \mid i \in U}$ means $a_j \gg 0$ for $j \notin U$.) Then
\[
\reg(I_{x_i \leq a_i \mid i=1,\ldots,n}) \leq \reg(I).
\]
\end{Claim}

Next lemma is an immediate consequence of Remark \ref{rem:3}
\begin{Lemma}\label{e1}
Let $I\subset S=\KK[x_1,x_2]$ be an equigenerated monomial ideal with $\mathcal{G}(I)=\{u_1,\ldots,u_m\}$, where $u_i=x_1^{\deg_{x_1}(u_i)}x_2^{\deg_{x_2}(u_i)}$  with  $\deg_{x_1}(u_1)<\deg_{x_1}(u_2)<\cdots<\deg_{x_1}(u_m)$ and $\deg_{x_1}(u_i)+\deg_{x_2}(u_i)=d$ for each $i\in [m]$. Then $\reg(I)=d$ if and only if $m=1$, or  $m\ge2$ and
$\deg_{x_1}(u_{i+1})-\deg_{x_1}(u_i)=1$ for each $i\in[m-1]$.
\end{Lemma}

\begin{Setting}\label{setting1}
Let $I\subset S=\KK[x_1,x_2,x_3]$ be an equigenerated monomial ideal of degree $d$. Let $I=\sum\limits_{i=1}^{t}I_i$ with  $t\ge 2$, where $I_i=(u_{i1},\ldots, u_{i\ell_i})x_3^{c_i}$ such that  $c_1<\cdots<c_t$ and $(u_{i1},\ldots, u_{i\ell_i})$ is an equigenerated monomial ideal in $\KK[x_1,x_2]$.  By changing the order of  the elements in $\{u_{i1},\ldots, u_{i\ell_i}\}$, we can assume that  each $u_{ij}=x_1^{\deg_{x_1}(u_{ij})}x_2^{\deg_{x_2}(u_{ij})}$  with  $0\le\deg_{x_1}(u_{i1})<\deg_{x_1}(u_{i2})<\cdots<\deg_{x_1}(u_{i\ell_i})$, $\deg_{x_2}(u_{i1})>\deg_{x_2}(u_{i2})>\cdots>\deg_{x_2}(u_{i\ell_i})\ge 0$  for any $i\in [t]$. Using Remark $\ref{rem:3} (1)$ and $(2)$, one can assume that $c_1=0$. For simplicity, denote $\deg_{x_i}(u_{jl})$ by $\d_{i,jl}$ or $\d_i(u_{jl})$.
\end{Setting}

First we consider the case where  $I=I_1+I_2$ has the form as in Setting \ref{setting1}.
\begin{Lemma}\label{e1.5}
Let  $I=I_1+I_2$  be an equigenerated monomial ideal of  degree $d$ as in Setting \ref{setting1}.
If $\reg(I)=d$, then
$\d_{i}(u_{1j})=\d_{i}(u_{2k})$ for some $i\in [2]$,  $j\in [\ell_1]$ and $k\in [\ell_2]$.
\end{Lemma}
\begin{proof}
One has $I_1 = I_{x_3 \leq 0}$, and by Claim \ref{claim:4}, $\reg(I_1) = d$.. By Lemma \ref{e1}, we have $\ell_1=1$, or  $\ell_1\ge2$ and $\d_{1}(u_{1(j+1)})-\d_{1}(u_{1j})=1$ for any $j\in[\ell_1-1]$.

In the following, we will show that $\d_{i}(u_{1j})=\d_{i}(u_{2k})$ for some $i\in [2]$,  $j\in [\ell_1]$ and $k\in [\ell_2]$.
By contradiction, assume that $\d_{i}(u_{1j})\neq\d_{i}(u_{2k})$ for any $i\in [2]$, $j\in[\ell_1]$ and $k\in[\ell_2]$. Then
$\d_{x_1}(u_{2k})<\deg_{x_1}(u_{11})$ or $\d_{x_1}(u_{2k})>\deg_{x_1}(u_{1\ell_1})$ for any $k\in [\ell_2]$, since $\d_{x_i}(u_{1j})\neq\d_{x_i}(u_{2k})$ for any $j\in[\ell_1]$. Similarly, $\d_{x_2}(u_{2k})>\d_{x_2}(u_{11})$ or $\d_{x_2}(u_{2k})<\d_{x_2}(u_{1\ell_1})$. Since $c_2>c_1$ and $\d_{1}(u_{1\ell_1})\ge \d_{1}(u_{11})$, both $\d_{1}(u_{2k})>\d_{1}(u_{1\ell_1})$ and $\d_{2}(u_{2k})>\d_{2}(u_{11})$ cannot occur simultaneously  for any $k\in [\ell_2]$. Otherwise, $\deg(u_{2k})=\d_{1}(u_{2k})+\d_{2}(u_{2k})>
\d_{1}(u_{1\ell_1})+\d_{2}(u_{11})\ge\d_{1}(u_{11})+\d_{2}(u_{11})=d-c_1$,  a contradiction.  Thus there are  three cases:

(i) If  $\d_{1}(u_{2k})<\d_{1}(u_{11})$ and $\d_{2}(u_{2k})<\d_{2}(u_{1\ell_1})$ for  some $k\in[\ell_2]$, then
we set $k_1=\min\{k\mid \d_{1}(u_{2k})<\d_{1}(u_{11})\text{\  and\ } \d_{2}(u_{2k})<\d_{2}(u_{1\ell_1})\}$ and $k_2=\max\{k\mid \d_{1}(u_{2k})<\d_{1}(u_{11})\text{ and\ } \d_{2}(u_{2k})<\d_{2}(u_{1\ell_1})\}$. Choose
 $J=(u_{2k_1},u_{2(k_1+1)},\ldots,u_{2k_2})x_3^{c_2}$, $K=J+(u)$, where $u=u_{11}x_3^{c_1}$,
then $J\cap (u)=(u_{11}x_3^{c_2})$ and  $K^{\calP}=J^{\calP}+(u)^{\calP}$  is
Betti splitting by  Lemma \ref{spliting}. It follows that
\begin{align*}
\reg(K)&=\reg(K^{\calP})=\max\{\reg(J^{\calP}),\reg((u)^{\calP}),\reg(K^{\calP}\cap (u)^{\calP})-1\}\\
&=\max\{\reg(J),\reg((u)),\reg(J\cap(u))-1\} \\
&\geq \reg(J\cap(u))-1=\d_{1}(u_{11})+\d_{2}(u_{11})+c_2-1\\
&\geq (\d_{1}(u_{2k})+1)+(\d_{2}(u_{2k})+1)+c_2-1=d+1.
 \end{align*}
By Claim \ref{claim:4},  we get that
$\reg(I)\geq\reg(K)\geq d+1$, a contradiction.

If (i) does not occur, there are the following two remaining case:

(ii) If $\d_{1}(u_{2k})<\d_{1}(u_{11})$ and $\d_{2}(u_{2k})>\d_{2}(u_{11})$ for some $k\in[\ell_2]$, then
we set $p=\max\{k\mid  \d_{1}(u_{2k})<\d_{1}(u_{11})$ and $\d_{2}(u_{2k})>\d_{2}(u_{11})\}$, $L=(u_{21},u_{22},\ldots,u_{2p})x_3^{c_2}$ and
 $M=L+(v)$, where  $v=u_{11}x_3^{c_1}$. Thus  $L\cap (v)=(x_1^{\d_{1}(u_{11})}x_2^{\d_{2}(u_{2p})}x_3^{c_2})$ and $M^{\calP}=L^{\calP}+(v)^{\calP}$ is Betti splitting.
This implies that $\reg(L\cap (v))=\d_{1}(u_{11})+\d_{2}(u_{2p})+c_2\geq \d_{1}(u_{11})+(\d_{2}(u_{11})+1)+(c_1+1)=d+2$. It follows from Lemma \ref{spliting} that
\[
\reg(M)=\max\{\reg(L),\reg((v)),\reg(L\cap(v))-1\}\geq d+1.
\]
By Claim \ref{claim:4}, we get that
$\reg(I)\geq\reg(M)\geq d+1$, a contradiction.

(iii)   Similarly to the proof of (ii), we can show that the case  $\d_{1}(u_{2k})>\d_{1}(u_{1\ell_1})$ and $\d_{2}(u_{2k})<\d_{2}(u_{1\ell_1})$  can not happen for any $k\in[\ell_2]$.
\end{proof}

\begin{Lemma}\label{e2}
Let  $I=I_1+I_2$  be an equigenerated monomial ideal  as in Setting \ref{setting1}. If $\reg(I)=d$, then $c_2-c_1=1$.
\end{Lemma}
\begin{proof}
  By Lemma \ref{e1.5}, $\d_{i}(u_{1j})=\d_{i}(u_{2k})$ for some $i\in [2]$, $j\in [\ell_1]$ and $k\in [\ell_2]$. By symmetry, we can assume that
  $i=1$, that is $\d_{1}(u_{1j})=\d_{1}(u_{2k})$, where $j\in[\ell_1]$ and $k\in[\ell_2]$.

 By contradiction, assume that $c_2\geq c_1+2$, then $\d_{2}(u_{2k})<\d_{2}(u_{1j})$.
 For the monomial $u_{1j}$, there are two cases:

  (i) If $\d_{2}(u_{2t})\neq\d_{2}(u_{1j})$ for all $t\in[\ell_2]$, then we set $t_1=\min\{t\mid \d_{2}(u_{2t})<\d_{2}(u_{1j})\}$, $J=(u_{2t_1},u_{2(t_1+1)},\ldots,u_{2k})x_3^{c_2}$
  and $K=J+(u)$, where $u=u_{1j}x_3^{c_{1}}$.
  Thus $J\cap (u)=(u_{1j}x_3^{c_{2}})$ and $\reg(J\cap(u))=\d_{1}(u_{1j})+\d_{2}(u_{1j})+c_{2}\geq \d_{1}(u_{1j})+\d_{2}(u_{1j})+c_{1}+2=d+2$.
Since  $K^{\calP}=J^{\calP}+(u)^{\calP}$  is Betti splitting, it follows from Lemma
\ref{spliting} that
\begin{align*}
\reg(K)&=\reg(K^{\calP})=\max\{\reg(J^{\calP}),\reg((u)^{\calP}),\reg(J^{\calP}\cap (u)^{\calP})-1\}\\
&=\max\{\reg(J),\reg((u)),\reg(J\cap(u))-1\} \\
&\geq d+1.
  \end{align*}
By Claim \ref{claim:4}, we get that
$\reg(I)\geq\reg(K)\geq d+1$, a contradiction.

(ii) If $\d_{2}(u_{2s})=\d_{2}(u_{1j})$ for some $s\in[\ell_2]$, then we need to consider the following three subcases:

(a) If $k-s=1$, then we set $L=(u_{2(k-1)}x_3^{c_2},u_{2k}x_3^{c_2},u_{1j}x_3^{c_1})$.
 Thus
 \[
 L=(1^{\d_{1}(u_{2(k-1)})}x_2^{\d_{2}(u_{2k})}x_3^{c_1})N,
 \]
 where $N=((x_2x_3)^{c_2-c_1},$ $(x_1x_3)^{c_2-c_1},(x_1x_2)^{c_2-c_1})$.
Since  $(N:x_3^{c_2-c_1})=(x_1^{c_2-c_1},x_2^{c_2-c_1})$ and $N+(x_3^{c_2-c_1})=((x_1x_2)^{c_2-c_1},x_3^{c_2-c_1})$, we get that $\reg((N:x_3^{c_2-c_1})(c_1-c_2))=2(c_2-c_1)-1+(c_2-c_1)=3(c_2-c_1)-1$ and $\reg(N+(x_3^{c_2-c_1}))=3(c_2-c_1)-1$. By Lemma \ref{exact}(1), we get that $\reg(N)=3(c_2-c_1)-1$. Thus $\reg(L)=(\d_{1}(u_{2(k-1)})+\d_{2}(u_{2k})+c_1)+\reg(N)=[(\d_{1}(u_{2k})+c_1-c_2)+\d_{2}(u_{2k})+c_1)]+(3(c_2-c_1)-1)=(\d_{1}(u_{2k})+\d_{2}(u_{2k})+c_2)+c_2-c_1-1\geq d+1$,
where the last inequality holds because $c_2\geq c_1+2$.
 By Claim \ref{claim:4}, 
$\reg(I)\ge \reg(L)\geq d+1$, a contradiction.

 (b)  If $k-s=\d_{1}(u_{2k})-\d_{1}(u_{2s})$, then $\d_{1}(u_{2(t+1)})-\d_{1}(u_{2t})=1$ for all $t\in \{s,\ldots, k-1\}$.
 Let $M=(u_{2s},u_{2(s+1)},\ldots,u_{2k})x_3^{c_2}$ and $N=M+(v)$, where $v=u_{1j}x_3^{c_1}$. Thus
 \[
 M=x_1^{\d_{1}(u_{2s})}x_2^{\d_{2}(u_{2k})}x_3^{c_2}\sum\limits_{t=0}^{k-s}(x_1^tx_2^{k-s-t})=x_1^{\d_{1}(u_{2s})}x_2^{\d_{2}(u_{2k})}x_3^{c_2}(x_1,x_2)^{k-s},
  \]
and $M\cap(v)=(u_{1j}x_3^{c_{2}})$. It follows from  Lemma \ref{regular} that
$\reg(M)=(\d_{1}(u_{2s})+\d_{2}(u_{2k})+c_{2})+(k-s)=\d_{1}(u_{2k})+\d_{2}(u_{2k})+c_{2}=d$ and $\reg(M\cap(v))=\d_{1}(u_{1j})+\d_{2}(u_{1j})+c_{2}\geq \d_{1}(u_{1j})+\d_{2}(u_{1j})+c_{1}+2=d+2$.
Since  $N^{\calP}=M^{\calP}+(v)^{\calP}$ is Betti splitting, similarly to the proof of (i), we can get that $\reg(I)\geq\reg(N)\geq d+1$, a contradiction.

 (c) If $1<k-s<\d_{1}(u_{2k})-\d_{1}(u_{2s})$, then there exists some  $t\in \{s,s+1,\ldots, k-1\}$ such that $\d_{1}(u_{2(t+1)})-\d_{1}(u_{2t})\geq 2$. Otherwise, $\d_{1}(u_{2k})-\d_{1}(u_{2s})=\sum\limits_{t=s}^{k-1}(\d_{1}(u_{2(t+1)})-\d_{1}(u_{2t}))=(k-1)-(s-1)=k-s$,  a contradiction.  Let $Q=(u_{2t},u_{2(t+1)})x_3^{c_2}$,  then $\reg(Q)=\d_{1}(u_{2(t+1)})+\d_{2}(u_{2t})+c_{2}-1\geq (\d_{1}(u_{2t})+2)+\d_{2}(u_{2t})+c_{2}-1=d+1$.  Similarly to the proof of (a), we can get $\reg(I)\ge \reg(Q)\geq d+1$, a contradiction.

 We complete the proof.
\end{proof}

For  an equigenerated monomial ideal  $I=I_1+I_2$   as in Setting \ref{setting1}, we will use the following conditions  to characterize when $\reg(I)=d$.
\begin{Definition}\label{condition1}
Let $I=I_1+I_2$ be an equigenerated monomial ideal as in Setting \ref{setting1}. We say that  $I$ satisfies  the conditions $(*)$ if  the following conditions hold

$(1^*)$  $c_2=c_1+1$;

$(2^*)$ $\reg(I_1)=d$;

$(3^*)$ there exist some  $i\in [2]$ such that $\d_{i}(u_{1j_i})=\d_{i}(u_{2k_i})$ for some $j_i\in [\ell_1]$ and $k_i\in [\ell_2]$;

$(4^*)$ if  $\ell_2\ge 2$ and  $\d_{1}(u_{2(j+1)})-\d_{1}(u_{2j})\geq 2$ for some  $j\in[\ell_2-1]$, then  $\d_{1}(u_{1m})=\d_{1}(u_{2(j+1)})$ and $\d_{2}(u_{1n})=\d_{2}(u_{2j})$  for some  $m,n\in[\ell_1]$.
\end{Definition}

Now we are ready to prove two main results of this section.
\begin{Lemma}\label{e4.1}
Suppose that $I=I_1+I_2$ is an equigenerated monomial ideal as in Setting \ref{setting1} satisfying condition $(*)$ of Definition \ref{condition1}.
Then $\reg(I)=d$.
\end{Lemma}
\begin{proof} For convenience, we can assume that $v_{1i}=u_{1i}x_3^{c_1}$ for each $i\in [\ell_1]$ and $v_{2i}=u_{2i}x_3^{c_2}$ for each $i\in [\ell_2]$.
By symmetry and the condition $(3^*)$ in Definition \ref{condition1}, we can also assume that
$\d_{1}(u_{1j_1})=\d_{1}(u_{2k_1})$ for some $j_1\in [\ell_1]$, $k_1\in [\ell_2]$. Thus
$v_{11},\ldots, v_{1\ell_1},v_{21},\ldots, v_{2\ell_2}$ is all of  minimal generators  of $I$ such that
$((v_{11},\ldots, v_{1i}):v_{1(i+1)})=(x_2)$ for all $i\in [\ell_1-1]$(if $\ell_1\ge 2$), and
\[
((v_{11},\ldots, v_{1\ell_1}):v_{2k_1})= \begin{cases}
					(x_2),& \text{if\ } \d_{2}(v)\geq\d_{2}(v_{2k_1})+1\ \text{for all\ }v\in\mathcal{G}(I_1),\\
(x_1,x_2),& \text{otherwise,}
				\end{cases}
 \]
since $c_2=c_1+1$. Therefore,  $I$ has linear quotients if  $\ell_2=1$, so $\reg(I)=d$.

In the following, we assume that  $\ell_2\ge 2$.
We divide into two cases:

(I) If  $\d_{1}(v_{2(i+1)})-\d_{1}(v_{2i})=1$ for every $i\in [\ell_2-1]$, then there are three subcases:

(a) If  $k_1=1$, then, for each $i\in[\ell_2-1]$, we have
\begin{align*}
&((v_{11},\ldots, v_{1\ell_1},v_{21}, v_{22},\ldots, v_{2i}):v_{2(i+1)})\\
=&\begin{cases}
					(x_2),&  \text{if\ } \d_{2}(v)\geq\d_{2}(v_{2(i+1)})+1\ \text{for all\ }v\in\mathcal{G}(I_1),\\
					(x_1,x_2),& \text{otherwise.}
				\end{cases}
\end{align*}

 (b) If $k_1=\ell_2$, then, for each $i\in[\ell_2-1]$, we have
\begin{align*}
&((v_{11},\ldots, v_{1\ell_1},v_{2\ell_2},v_{2(\ell_2-1)},  \ldots, v_{2(\ell_2-i+1)}):v_{2(\ell_2-i)})\\
&=\begin{cases}
					(x_1), & \text{if\ } \d_{1}(v)\geq\d_{1}(v_{2(\ell_2-i)})+1\ \text{for all\ }v\in\mathcal{G}(I_1),\\
					(x_1,x_2),& \text{otherwise.}
				\end{cases}
\end{align*}

(c) If $1< k_1<\ell_2$, then, for each $k_1\le i\le \ell_2-1$, we have
\begin{align*}
&((v_{11},\ldots, v_{1\ell_1},v_{2k_1}, v_{2(k_1+1)},\ldots, v_{2i}):v_{2(i+1)})\\
&=\begin{cases}
	(x_2), & \text{if\ } \d_{2}(v)\geq\d_{2}(v_{2(i+1)})+1\ \text{for all\ }v\in\mathcal{G}(I_1),\\
	(x_1,x_2), & \text{otherwise,}
\end{cases}
\end{align*}
and  for each $i\in[k_1-1]$, we also have
\begin{align*}
&((v_{11},\ldots, v_{1\ell_1},v_{2k_1}, v_{2(k_1+1)}, \ldots, v_{2\ell_2},v_{2(k_1-1)},v_{2(k_1-2)},\ldots,v_{2(k_1-i+1)}):v_{2(k_1-i)})\\
&=\begin{cases}
	(x_1), & \text{if\ } \d_{1}(v)\geq\d_{1}(v_{2(k_1-i)})+1\ \text{for all\ }v\in\mathcal{G}(I_1),\\
	(x_1,x_2), & \text{otherwise.}
\end{cases}
\end{align*}
In either of the three cases above, we get that $I$  has linear quotients, so $\reg(I)=d$.

(II) If $\d_{1}(v_{2(i+1)})-\d_{1}(v_{2i})\ge 2$ for some  $i\in[\ell_2-1]$. Let
\[
s=|\{i\mid \d_{1}(v_{2(i+1)})-\d_{1}(v_{2i})\ge 2\}|.
\]
 Using induction on $s$, we will show  that $I$ has linear quotients, which implies that $\reg(I)=d$.

 If $s=1$, then $\d_{1}(v_{2(p+1)})-\d_{1}(v_{2p})\ge 2$ for  a unique  $p\in [\ell_2-1]$ and $\d_{1}(v_{2(i+1)})-\d_{1}(v_{2i})=1$ for all other $i\in[\ell_2-1]$. According to the condition $(4^*)$ in Definition \ref{condition1}, there exists some $q,r\in [\ell_1]$ such that  $\d_{1}(u_{1q})=\d_{1}(u_{2(p+1)})$ and $\d_{2}(u_{1r})=\d_{2}(u_{2p})$. From the proof of the case (I) above,
 we know that the ideal $(v_{11},\ldots, v_{1\ell_1}, v_{2(p+1)}, v_{2(p+2)},\ldots, v_{2\ell_2})$ has linear quotients.
  If  $p=1$, then
\begin{align*}
&((v_{11},\ldots, v_{1\ell_1},v_{22}, v_{23},\ldots, v_{2\ell_2}):v_{21})\\
=&\begin{cases}
					(x_1),& \text{if\ } \d_{1}(v)\geq\d_{1}(v_{2p})+1\ \text{for all\ }v\in\mathcal{G}(I_1),\\
(x_1,x_2),& \text{otherwise.}
				\end{cases}
\end{align*}
If $p\ge 2$, then
\begin{align*}
&((v_{11},\ldots, v_{1\ell_1},v_{2(p+1)}, v_{2(p+2)},\ldots, v_{2\ell_2}):v_{2p})\\
=& \begin{cases}
					(x_1),& \text{if\ } \d_{1}(v)\geq\d_{1}(v_{2p})+1\ \text{for all\ }v\in\mathcal{G}(I_1),\\
(x_1,x_2),& \text{otherwise.}
				\end{cases}
\end{align*}
and for each $i\in[p-1]$, we  have
\begin{align*}
&((v_{11},\ldots, v_{1\ell_1},v_{2(p+1)}, v_{2(p+2)},\ldots, v_{2\ell_2},v_{2p},v_{2(p-1)},\ldots, v_{2(p-i+1)}):v_{2(p-i)})\\
=& \begin{cases}
					(x_1),& \text{if\ } \d_{1}(v)\geq\d_{1}(v_{2(p-i)})+1\ \text{for all\ }v\in\mathcal{G}(I_1),\\
(x_1,x_2),& \text{otherwise.}
				\end{cases}
\end{align*}
Therefore,  $I$ has linear quotients.

If $s\geq 2$, then by the definition of $s$, there exist $1\le q_{1}<q_{2}<\cdots<q_{s}\le \ell_1$, $1\le r_{1}<r_{2}<\cdots<r_{s}\le \ell_1$
 such that $\d_{1}(u_{1q_{\ell}})=\d_{1}(u_{2(p_{\ell}+1)})$ and $\d_{2}(u_{1r_{\ell}})=\d_{2}(u_{2p_{\ell}})$.
By induction,  the ideal $(v_{11},\ldots, v_{1\ell_1}, v_{2(p_s+1)},\ldots,v_{2\ell_2}$, $v_{2p_s},v_{2(p_s-1)},\ldots, v_{2(p_1+1)})$ has linear quotients.
If  $p_1=1$, then
\begin{align*}
&((v_{11},\ldots, v_{1\ell_1},v_{2(p_s+1)},v_{2(p_s+2)},\ldots, v_{2\ell_2},v_{2p_s},v_{2(p_s-1)},v_{2(p_s-2)},\ldots, v_{22}):v_{21})\\
=& \begin{cases}
					(x_1),& \text{if\ } \d_{1}(v)\geq\d_{1}(v_{2p_1})+1\ \text{for all\ }v\in\mathcal{G}(I_1),\\
(x_1,x_2),& \text{otherwise.}
				\end{cases}
\end{align*}
 If $p_1\ge 2$, then for each $i\in[p_1-1]$, we have
\begin{align*}
&((v_{11},\ldots, v_{1\ell_1},v_{2(p_s+1)},\ldots, v_{2\ell_2},v_{2p_s},v_{2(p_s-1)},v_{2(p_s-2)},\ldots,v_{2(p_1-i+1)}):v_{2(p_1-i)})\\
=& \begin{cases}
					(x_1),& \text{if\ } \d_{1}(v)\geq\d_{1}(v_{2(p_1-i)})+1\ \text{for all\ }v\in\mathcal{G}(I_1),\\
(x_1,x_2),& \text{otherwise.}
				\end{cases}
\end{align*}
Therefore, $I$ also has linear quotients.  We complete the proof.
\end{proof}

\begin{Lemma}\label{e4}
Let $I=I_1+I_2$ be an equigenerated monomial ideal, as defined in Setting \ref{setting1}.  If  $\reg(I)=d$, then $I$ satisfies the conditions $(*)$ in Definition \ref{condition1}.
\end{Lemma}

\begin{proof}
Conditions $(1^*)$, $(2^*)$ and $(3^*)$ follow from  Lemmas  \ref{e2}, \ref{e1},  and \ref{e1.5}, respectively. By Lemma \ref{e1} This implies that $\ell_1=1$, or $\ell_1\ge 2$ and $\d_{1}(u_{1(i+1)})=\d_{1}(u_{1i})+1$ for each $i\in[\ell_1-1]$.

Next, we will prove the condition $(4^*)$ by contradiction, dividing  the proof  into three cases:

(i) If $\d_{2}(u_{1n})=\d_{2}(u_{2j})$ for some $n\in[\ell_1]$, but  $\d_{1}(u_{1m})\neq \d_{1}(u_{2(j+1)})$ for any $m\in[\ell_1]$, then
we consider the ideal  $J=K+(v)$, where
\[
K=(u_{1n},u_{1(n+1)},\ldots,u_{1\ell_1})x_3^{c_1}+(u_{2j}x_3^{c_2})\ \text{\ and\ }\ v=u_{2(j+1)}x_3^{c_2}.
\]
 Thus $K\cap (v)=(x_1^{\d_{1}(u_{2(j+1)})}x_2^{\d_{2}(u_{1\ell_1})}x_3^{c_2})$, since $\d_{1}(u_{2j})+1=(d-\d_{2}(u_{2j})-c_2)+1=d-\d_{2}(u_{1n})-c_1=\d_{1}(u_{1n})<\d_{1}
 (u_{1(n+1)})<\cdots<\d_{1}(u_{1\ell_1})<\d_{1}(u_{2(j+1)})$ and $\d_{2}(u_{2j})=\d_{2}(u_{1n})>\d_{2}(u_{1(n+1)})>\cdots>\d_{2}(u_{1\ell_1})>\d_{2}(u_{2(j+1)})$, where $\d_{1}(u_{1\ell_1})<\d_{1}(u_{2(j+1)})$  is due to
  $\d_{1}(u_{2(j+1)})\ge \d_{1}(u_{2j})\\ +1$, $\d_{1}(u_{1(i+1)})=\d_{1}(u_{1i})+1$ for each $i\in[\ell_1-1]$ and $\d_{1}(u_{2(j+1)})\neq \d_{1}(u_{1i})$ for all $i\in[\ell_1]$, and
 $\d_{2}(u_{1\ell_1})>\d_{2}(u_{2(j+1)})$ is due to  $\d_{2}(u_{1\ell_1})=d-\d_{1}(u_{1\ell_1})-c_1>d-\d_{1}(u_{2(j+1)})-c_1=d-\d_{1}(u_{2(j+1)})-c_2+1=\d_{2}(u_{2(j+1)})+1$.
  It follows that  $\reg(K\cap (v))=\d_{1}(u_{2(j+1)})+\d_{2}(u_{1\ell_1})+c_2\geq (\d_{1}(u_{1\ell_1})+1)+\d_{2}(u_{1\ell_1})+(c_1+1)=d+2$, $\reg((v))=d$, and $\reg(K)=d$ by Lemma \ref{e4.1}. From Lemma $\ref{exact}(2)$
 the following exact sequence
	\begin{equation}
		0\longrightarrow \frac{S}{K \cap (v)}\longrightarrow \frac{S}{K} \oplus
\frac{S}{(v)}\longrightarrow \frac{S}{J} \longrightarrow 0,
		\label{eqn:SES-4}
	\end{equation}	
we can get $\reg(J)\geq d+1$.

By Claim \ref{claim:4}, we get that
$\reg(I)\ge \reg(J)\geq d+1$, a contradiction.

(ii) If $\d_{1}(u_{1m})=\d_{1}(u_{2(j+1)})$ for some $m\in[\ell_1]$,  but   $\d_{2}(u_{1n})\ne \d_{2}(u_{2j})$ for any $n\in[\ell_1]$, then we consider the ideal  $J=K+(v)$, where
\[
K=(u_{11},u_{12},\ldots,u_{1m})x_3^{c_1}+(u_{2(j+1)}x_3^{c_2})\ \text{\ and\ }\ v=u_{2j}x_3^{c_2}.
\]
 By similar arguments to the proof of (i), we  can deduce that $\reg(J)\geq d+1$ and $\reg(I)\ge \reg(J)\geq d+1$, a contradiction.

(iii) If  $\d_{1}(u_{1m})\neq \d_{1}(u_{2(j+1)})$  and $\d_{2}(u_{1n})\neq\d_{2}(u_{2j})$ for all $m,n\in[\ell_1]$, then, in particular,   $\d_{1}(u_{11})\neq\d_{1}(u_{2(j+1)})$ and $\d_{2}(u_{1\ell_1})\neq\d_{2}(u_{2j})$.
 So there are two subcases that need to be considered:

(a) If $\d_{1}(u_{11})>\d_{1}(u_{2(j+1)})$ or $\d_{2}(u_{1\ell_1})>\d_{2}(u_{2j})$, then we choose
\begin{align*}
M&=(u_{2j},u_{2(j+1)})x_3^{c_2}\\
&=(x_2^{\d_{2}(u_{2j})-\d_{2}(u_{2(j+1)})},x_1^{\d_{1}(u_{2(j+1)})-\d_{1}(u_{2j})})
x_1^{\d_{1}(u_{2j})}x_2^{\d_{2}(u_{2(j+1)})}x_3^{c_2}.
\end{align*}
 Thus $\reg(M)=(\d_{1}(u_{2j})+\d_{2}(u_{2(j+1)})+c_2)+(\d_{2}(u_{2j})-\d_{2}(u_{2(j+1)}))+
(\d_{1}(u_{2(j+1)})-\d_{1}(u_{2j}))-1=\d_{2}(u_{2j})+\d_{1}(u_{2(j+1)})+c_2-1\geq \d_{1}(u_{2j})+(\d_{1}(u_{2j})+2)+c_2-1=d+1$. By Claim \ref{claim:4}, we get that
$\reg(I)\ge \reg(M)\geq d+1$, a contradiction.

(b) If $\d_{1}(u_{11})<\d_{1}(u_{2(j+1)})$ and $\d_{2}(u_{1\ell_1})<\d_{2}(u_{2j})$, then $\d_{2}(u_{11})<\d_{2}(u_{2j})$. Otherwise,  $\d_{2}(u_{11})\geq \d_{2}(u_{2j})>\d_{2}(u_{1\ell_1})$, this forces $\d_{2}(u_{2j})=\d_{2}(u_{1i})$ for some $i$, which  contradicts the hypothesis that $\d_{2}(u_{1n})\neq\d_{2}(u_{2j})$ for all $n\in[\ell_1]$. Let
\begin{align*}
N=&(u_{2j}x_3^{c_2},u_{11}x_3^{c_1})\\
=&(x_2^{\d_{x_2}(u_{2j})-\d_{2}(u_{11})}x_3^{c_2-c_1},
x_1^{\d_{1}(u_{11})-\d_{1}(u_{2j})})x_1^{\d_{1}(u_{2j})}x_2^{\d_{2}(u_{11})}x_3^{c_1},
\end{align*}
then $\reg(N)=\d_{2}(u_{2j})+\d_{1}(u_{11})+c_2-1\geq (\d_{1}(u_{11})+1)+\d_{1}(u_{11})+(c_1+1)-1=d+1$.
By Claim \ref{claim:4}, we get that
$\reg(I)\ge \reg(N)\geq d+1$, a contradiction.
\end{proof}

Next, we consider the case where  $I=\sum\limits_{i=1}^{t}I_i$ has the form given as in Setting \ref{setting1}, where   $t\ge 3$.  We obtain
\begin{Lemma}\label{e4.5}
Let  $I=\sum\limits_{i=1}^{t}I_i$ be an equigenerated monomial ideal as in Setting \ref{setting1}, where  $t\ge 3$. If there exists  some $k\in \{2,\ldots,t\}$ such that $\ell_k\ge 2$ and  $\d_{1}(u_{k(j+1)})-\d_{1}(u_{kj})\geq 2$ for some  $j\in[\ell_k-1]$, and
 $\reg(I)=d$. Then  $\d_{1}(u_{(k-1)m_k})=\d_{1}(u_{k(j+1)})$ and $\d_{2}(u_{(k-1)n_k})=\d_{2}(u_{kj})$  for some  $m_k,n_k\in[\ell_{k-1}]$.
\end{Lemma}
\begin{proof}
By symmetry, let $\d_{1}(u_{(k-1)m_k})\neq \d_{1}(u_{k(j+1)})$ for any $m_k\in[\ell_{k-1}]$, then
we consider the set of monomials in $\mathcal{G}(I)$
\[
T=\{h\in \mathcal{G}(I)\mid \d_{1}(h)<\d_{1}(u_{k(j+1)}) \text{\ and\ } \d_{3}(h)\leq c_{k}\}.
\]
It is obvious that  $T\ne \emptyset$, since $u_{kj}x_3^{c_k}\in T$.

Since $\d_{1}(u_{k(j+1)})-\d_{1}(u_{kj})\geq 2$ and $\d(u_{k(j+1)})=\d(u_{kj})$, we conclude that $\d_{2}(u_{kj})-\d_{2}(u_{k(j+1)})\geq 2$. There are two cases to consider:

(i) If there exists some $h\in T$ such that $\d_{2}(u_{k(j+1)})<\d_{2}(h)<\d_{2}(u_{kj})$, then we set
\[
W=\{h\in T\mid \d_{2}(u_{k(j+1)})<\d_{2}(h)<\d_{2}(u_{kj})\}
\]
 and $L=(W)+(u_{k(j+1)}x_3^{c_k})$.
Thus $(W)\cap (u_{k(j+1)}x_3^{c_k})=(x_1^{\d_{1}(u_{k(j+1)})}x_2^bx_3^{c_k})$, where $b=\min\{\d_{2}(h)\mid h\in W\}$. Also,  $\reg((u_{k(j+1)}x_3^{c_k}))=d$ and $\reg((W))=d$ by Claim \ref{claim:4}.
From the choice of $h$, we know that $b=d-\d_{1}(h)-\d_{3}(h)\geq d-(\d_{1}(u_{k(j+1)})-1)-(c_k-1)=\d_{2}(u_{k(j+1)})+2>\d_{2}(u_{k(j+1)})+1$, where $\d_{3}(h)\leq c_k-1$ is due to $\d_{2}(u_{k(j+1)})<\d_{2}(h)<\d_{2}(u_{kj})$. This forces that $\reg((W)\cap(u_{k(j+1)}x_3^{c_k}))=\d_{1}(u_{k(j+1)})+b+c_k>d+1$.
 From Lemma $\ref{exact}(2)$ and the following exact sequence
	\begin{equation}
		0\longrightarrow \frac{S}{(W) \cap (u_{k(j+1)}x_3^{c_k})}\longrightarrow \frac{S}{(W)} \oplus
\frac{S}{(u_{k(j+1)}x_3^{c_k})}\longrightarrow \frac{S}{L} \longrightarrow 0,
		\label{eqn:SES-4}
	\end{equation}	
we can get $\reg(L)\geq d+1$. By Claim \ref{claim:4}, we get that
$\reg(I)\ge \reg(L)\geq d+1$, a contradiction.

(ii) If  there is no $h\in T$ such that $\d_{2}(u_{k(j+1)})<\d_{2}(h)<\d_{2}(u_{kj})$, then
we consider the set of monomials in $\mathcal{G}(I)$
\[
W'=\{h\in T\mid \d_{2}(h)=\d_{2}(u_{kj})\}.
\]
Let $N=(W')+(u_{k(j+1)}x_3^{c_k})$, then  $\reg(N)=d$ by Claim \ref{claim:4}.
For any $h\in (W')$, we have $\d_{2}(h)=\d_{2}(u_{kj})=d-\d_{1}(u_{kj})-c_k\geq d-(\d_{1}(u_{k(j+1)})-2)-c_k=\d_{2}(u_{k(j+1)})+2$.
   Lemma \ref{e4} (1) implies that $\reg(N)\neq d$, which is  a contradiction.

We have completed the proof.
\end{proof}

\begin{Remark}\label{3} Let $I\subset S=\KK[x_1,x_2,x_3]$ be an equigenerated monomial ideal of degree $d$. We can also   write $I$  as $I=\sum\limits_{i=1}^{s}J_i$ with $s\ge 2$, where $J_i=(v_{i1},\ldots, v_{ip_i})x_k^{d_{ik}}$ such that  $d_{1k}<\cdots<d_{sk}$ and $(v_{i1},\ldots, v_{ip_i})$ is an equigenerated monomial ideal in $\KK[x_{\alpha},x_{\beta}]$, where  $\alpha, \beta\in [3]\setminus \{k\}$.  We can also  assume that  each  $v_{ij}=x_{\alpha}^{\d_{{\alpha}}(v_{ij})}x_{\beta}^{\d_{{\beta}}(v_{ij})}$ are distinct and satisfy $0\le \d_{{\alpha}}(v_{i1})<\d_{{\alpha}}(v_{i2})<\cdots<\d_{{\alpha}}(v_{ip_i})$ and $\d_{{\beta}}(v_{i1})>\d_{{\beta}}(v_{i2})>\cdots>\d_{{\beta}}(v_{ip_i})\ge 0$ for any $i\in [s]$. If there exists some $v_{ij}=x_{\alpha}^{\d_{{\alpha}}(v_{ij})}x_{\beta}^{\d_{{\beta}}(v_{ij})}$, $v_{i(j+1)}=x_{\alpha}^{\d_{{\alpha}}(v_{i(j+1)})}x_{\beta}^{\d_{{\beta}}(v_{i(j+1)})}$  such that $\d_{{\alpha}}(v_{i(j+1)})-\d_{{\alpha}}(v_{ij})\geq 2$ and
 $\reg(I)=d$, then similar to Lemma \ref{e4.5}, we can conclude that
  $\d_{{\alpha}}(v_{(i-1)m'_i})=\d_{{\alpha}}(v_{i(j+1)})$ and $\d_{{\beta}}(v_{(i-1)n'_i})=\d_{{\beta}}(v_{ij})$  for some  $m'_i,n'_i\in[p_{i-1}]$.
\end{Remark}

Similarly to Definition \ref{condition1}, we define
\begin{Definition}\label{condition2}
Let $I=\sum\limits_{i=1}^{t}I_i$ be an equigenerated monomial ideal as in Setting \ref{setting1}. We say that  $I$ satisfies  the conditions $(**)$ if  the following conditions hold

$(1^{**})$ $\reg(\sum\limits_{i=1}^{\ell}I_i)=d$ where $\ell\in [t-1]$;

$(2^{**})$ For any $r\in[t-1]$, there exist some  $i\in [2]$ such that $\d_{i}(u_{(r+1)j_i})=\d_{i}(u_{rk_i})$ $\d_{i}(u_{(r+1)j_i})=\d_{i}(u_{rk_i})$ for some $j_i\in [\ell_{r+1}]$ and $k_i\in [\ell_r]$;

$(3^{**})$ $c_{\zeta+1}-c_{\zeta}=1$ for any $\zeta\in[t-1]$;

$(4^{**})$ If $I$ has the form  $I=\sum\limits_{i=1}^{s}J_i$ as in Remark \ref{3}, and $v_{ij}=x_{\alpha}^{\d_{{\alpha}}(v_{ij})}x_{\beta}^{\d_{{\beta}}(v_{ij})}$, $v_{i(j+1)}=x_{\alpha}^{\d_{{\alpha}}(v_{i(j+1)})}x_{\beta}^{\d_{{\beta}}(v_{i(j+1)})}$  such that $\d_{{\alpha}}(v_{i(j+1)})-\d_{{\alpha}}(v_{ij})\geq 2$.
Then  $\d_{{\alpha}}(v_{(i-1)m'_i})=\d_{{\alpha}}(v_{i(j+1)})$ and $\d_{{\beta}}(v_{(i-1)n'_i})=\d_{x_{\beta}}(v_{ij})$  for some  $m'_i,n'_i\in[p_{i-1}]$. Furthermore, we also have  $\d_{x_{\alpha}}(u)> \d_{x_{\alpha}}(v_{ij})$ or $\d_{x_{\beta}}(u)\ge \d_{x_{\beta}}(v_{ij})$ for any $u\in\mathcal{G}(I)$.
\end{Definition}

Next, we are ready to prove another two main results of this section.
\begin{Lemma}\label{e5}
Let $I\subset S=\KK[x_1,x_2,x_3]$ be an equigenerated monomial ideal  as in Setting \ref{setting1}, where  $t\ge 3$.   If  $\reg(I)=d$, then $I$ satisfies the conditions $(**)$ in Definition \ref{condition2}.
\end{Lemma}
\begin{proof}
First, we will prove  the condition $(1^{**})$  by descending induction on $\ell$. If $\ell=t-1$, then
we choose  $J=\sum\limits_{i=1}^{t-1}I_i$, By Claim \ref{claim:4}, one has
 $\reg(J)=d$. Now, assume that $\ell\le t-1$ and $\reg(\sum\limits_{i=1}^{\ell}I_i)=d$. Choose $J_1=\sum\limits_{i=1}^{\ell-1}I_i$.  By Claim \ref{claim:4}, we get that
 $\reg(\sum\limits_{i=1}^{\ell}I_i)\geq\reg(J_1)\geq d$, which implies $\reg(\sum\limits_{i=1}^{\ell-1}I_i)=d$.

$(2^{**})$ For any $r\in[t-1]$, we distinguish between  the following two cases:

 (a) If $\ell_{r+1}\ge 2$ and  $\d_{1}(u_{(r+1)(j+1)})-\d_{1}(u_{(r+1)j})\geq 2$ for some  $j\in[\ell_{r+1}-1]$, then the desired result follows from  Lemma \ref{e4.5}.

 (b) If $\ell_{r+1}=1$, or  $\ell_{r+1}\ge 2$ and  $\d_{1}(u_{(r+1)(j+1)})-\d_{1}(u_{(r+1)j})=1$ for any  $j\in[\ell_{r+1}-1]$, then  by similar arguments as the proof of Lemma \ref{e1.5}, we can obtain the desired result.

$(3^{**})$  If $c_{\zeta+1}-c_{\zeta}\ge 2$, then,  by symmetry and the condition $(2^{**})$, we can assume that
 $\d_{1}(u_{(\zeta+1)j_1})=\d_{1}(u_{\zeta k_1})$ for some $j_1\in [\ell_{\zeta+1}]$ and $k_1\in [\ell_{\zeta}]$. Thus $u_{(\zeta+1)j_1}x_3^{c_{\zeta+1}}$ and $u_{\zeta k_1}x_3^{c_{\zeta}}$ be written as $u_{(\zeta+1)j_1}x_3^{c_{\zeta+1}}=v_{qr}x_1^{d_{q1}}$ and  $u_{\zeta k_1}x_3^{c_{\zeta}}=v_{q(r+1)}x_1^{d_{q1}}$, respectively, where $v_{qr},v_{q(r+1)}\in\KK[x_2,x_3]$. Since  $\d_{3}(u_{(\zeta+1)j_1})-\d_{3}(u_{\zeta k_1})\\ =c_{\zeta+1}-c_{\zeta}\ge 2$, by Remark \ref{3}, there exist two monomials $v_{(q-1)m_q'},v_{(q-1)n_q'}\in\KK[x_2,x_3]$ such that $\d_{2}(v_{(q-1)m_q'})=\d_{2}(v_{q(r+1)})=\d_{2}(u_{\zeta k_1})$ and $\d_{3}(v_{(q-1)n_q'})=\d_{3}(v_{qr})=c_{\zeta+1}$.

 Now, we will prove that $v_{(q-1)m_q'}=v_{(q-1)n_q'}$. It suffices to show that $\d_{3}(v_{(q-1)m_q'})=\d_{x_3}(v_{(q-1)n_q'})$,  since $\deg(v_{(q-1)m_q'})=\deg(v_{(q-1)n_q'})$. By Remark  \ref{3}, $d_{(q-1)1}\le  d_{q1}-1$. Thus  $c_{\zeta+1}\ge \d_{3}(v_{(q-1)m_q'})=d-d_{(q-1)1}-\d_{2}(v_{(q-1)m_q'})\ge d-(d_{q1}-1)-\d_{2}(v_{(q-1)m_q'})= d-(d_{q1}-1)-\d_{2}(v_{q(r+1)})=\d_{3}(v_{q(r+1)})+1=c_{\zeta}+1$, which forces $\d_{3}(v_{(q-1)m_q'})=c_{\zeta+1}$. The desired result follows from the fact that
  $\d_{3}(v_{(q-1)n_q'})=c_{\zeta+1}$.

Next, we compute the regularity of  $K=(v_{(q-1)m_q'}x_1^{d_{(q-1)1}})+(v_{qr},v_{q(r+1)})x_1^{d_{q1}}$.  By Claim \ref{claim:4}, then $\reg(K)=d$. It follows from Lemma \ref{e4} that $d_{(q-1)1}=d_{q1}-1$. However,
 $d_{(q-1)1}=\d_{1}(v_{(q-1)m_q'}x_1^{d_{(q-1)1}})=d-\d_{2}(v_{(q-1)m_q'})-\d_{3}(v_{(q-1)m_q'})=d-\d_{2}(u_{\zeta k_1})-c_{\zeta+1}\leq d-\d_{2}(u_{\zeta k_1})-(c_{\zeta}+2)=d-\d_{2}(u_{\zeta k_1})-\d_{3}(u_{\zeta k_1})-2=\d_{1}(u_{\zeta k_1})-2=d_{q1}-2$. It's impossible.

$(4^{**})$  The first part follows from  Remark \ref{3}.  Now, we will prove the second part.
If there exists some $u\in\mathcal{G}(I)$ such that  $\d_{{\alpha}}(u)\le \d_{{\alpha}}(v_{ij})$ and $\d_{{\beta}}(u)<\d_{{\beta}}(v_{ij})$, then
we set
\begin{align*}
i_1&=\min\{i'\mid \d_{{\alpha}}(v_{i'(j_{i'}+1)})-\d_{{\alpha}}(v_{i'j_{i'}})\geq 2 \text{\ and\ } \d_{{\beta}}(v_{i'j_{i'}})=\d_{{\beta}}(v_{ij})\},\\
a&=\min\{\d_{k}(v)\mid \d_{{\alpha}}(v)\leq \d_{{\alpha}}(v_{i_1j_{i_1}})\text{\ and\ } \d_{{\beta}}(v)<\d_{{\beta}}(v_{i_1j_{i_1}})\},\\
b&=\min\{\d_{{\alpha}}(v)\mid \d_{{k}}(v)=a\text{\ and\ }\d_{{\beta}}(v)<\d_{{\beta}}(v_{i_1j_{i_1}})\}.
\end{align*}
Since $\d_{{\alpha}}(v_{(i_1-1)m_{i_1}'})=\d_{{\alpha}}(v_{i_1(j_{i_1}+1)})$, $\d_{{\beta}}(v_{(i_1-1)n_{i_1}'})=\d_{{\beta}}(v_{i_1j_{i_1}})$ for some $m_{i_1}',n_{i_1}'\in[\ell_{i_1-1}]$,
we have $n_{i_1}'<m_{i_1}'$.  By minimality of $i_1$,
\[
v_{(i_1-1)(n'_{i_1}+1)}x_{k}^{d_{(i_1-1)k}}=x_{\alpha}^{\d_{{\alpha}}(v_{(i_1-1)(n'_{i_1}+1)})}
x_{\beta}^{\d_{{\beta}}(v_{(i_1-1)(n'_{i_1}+1)})}x_{k}^{d_{(i_1-1)k}}
\]
satisfies $\d_{{\alpha}}(v_{(i_1-1)(n'_{i_1}+1)})=\d_{{\alpha}}(v_{(i_1-1)n'_{i_1}})+1$,
 that is, \[v_{(i_1-1)(n'_{i_1}+1)}x_{k}^{d_{(i_1-1)k}}=x_{\alpha}^{\d_{{\alpha}}(v_{i_1j_{i_1}})+2}
 x_{\beta}^{\d_{{\beta}}(v_{i_1j_{i_1}})-1}x_k^{d_{i_1k}-1}.\]

 Let $K$ and $L$ be two ideals with $\mathcal{G}(K)=\{v\in\mathcal{G}(I)\mid \d_{{\alpha}}(v)\leq \d_{{\alpha}}(v_{i_1j_{i_1}})+2 \text{\and\ }\d_{{\beta}}(v)<\d_{{\beta}}(v_{i_1j_{i_1}})\}$ and $\mathcal{G}(L)=\{v\in\mathcal{G}(I)\mid \d_{{\alpha}}(v)\leq \d_{{\alpha}}(v_{i_1j_{i_1}})+2, \d_{{\beta}}(v)=\d_{{\beta}}(v_{i_1j_{i_1}}) \text{\and\ } \d_{{k}}(v)\leq a\}=\{v_{wq}x_{\beta}^{\d_{{\beta}}(v_{i_1j_{i_1}})},v_{w(q+1)}x_{\beta}^{\d_{{\beta}}(v_{i_1j_{i_1}})},\\
 \ldots,v_{w(q+\ell)}x_{\beta}^{\d_{{\beta}}(v_{i_1j_{i_1}})}\}$,
 where $v_{wp}=x_{\alpha}^{\d_{{\alpha}}(v_{wp})}x_{k}^{\d_{{k}}(v_{wp})}\in\KK[x_{\alpha},x_{\beta}]$, $\alpha, k\in [3]\setminus \{\beta\}$ are distinct and $\d_{{\alpha}}(v_{wq})<\d_{{\alpha}}(v_{w(q+1)})<\cdots<\d_{{\alpha}}(v_{w(q+\ell)})$.
Then, for any $p\in [q+\ell-1]\setminus [q]$, $\d_{{\alpha}}(v_{w(p+1)})-\d_{{\alpha}}(v_{wp})=1$, which implies  $\reg(L)=d$ by Lemma \ref{regular}.
In fact, if $\d_{{\alpha}}(v_{w(p+1)})-\d_{{\alpha}}(v_{wp})\geq 2$ for  some $p\in [q+\ell-1]\setminus [q]$, then by Remark \ref{3},  $\d_{{\alpha}}(v_{(w-1)m_w})=\d_{{\alpha}}(v_{w(p+1)})$ and $\d_{{k}}(v_{(w-1)n_w})=\d_{{k}}(v_{wp})$ for some $m_w,n_w\in[\ell_{s-1}]$. Thus $\deg(v_{(w-1)m_w})=\deg(v_{wp})+1=d-(\d_{{\beta}}(v_{i_1j_{i_1}})-1)$ by the condition $(3^{**})$, which forces that
 $\d_{{k}}(v_{(w-1)m_w})=[d-(\d_{{\beta}}(v_{i_1j_{i_1}})-1)]-\d_{{\alpha}}(v_{(w-1)m_w})=
 (d-\d_{{\beta}}(v_{i_1j_{i_1}})+1)-\d_{{\alpha}}(v_{w(p+1)})\le (d-\d_{{\beta}}(v_{i_1j_{i_1}})+1)-(\d_{{\alpha}}(v_{wp})+2)=\d_{{k}}(v_{wp})-1\leq a-1$,
where the first  inequality holds because $\d_{{\alpha}}(v_{w(p+1)})-\d_{{\alpha}}(v_{wp})\ge 2$.
This contradicts with the  minimality  of $a$.

Let $b'=\max\{\d_{{\alpha}}(v_{wq}),b\}$, then $b'\leq \d_{{\alpha}}(v_{i_1j_{i_1}})$. In fact, from  the relationship $\d_{{\alpha}}(v_{wq})<\d_{{\alpha}}(v_{w(q+1)})<\cdots<\d_{{\alpha}}(v_{w(q+\ell)})$ and the   definition of $b$, we can conclude that $\d_{{\alpha}}(v_{wq})\leq \d_{{\alpha}}(v_{i_1j_{i_1}})$ and $b\leq \d_{{\alpha}}(v_{i_1j_{i_1}})$.
 This implies that
 \begin{align*}
 L\cap K=&(x_{\alpha}^{\d_{{\alpha}}(v_{i_1j_{i_1}})+2}x_{\beta}^{\d_{{\beta}}(v_{i_1j_{i_1}})}x_k^{d_{i_1k}-1},x_{\alpha}^{b'}x_{\beta}^{\d_{{\beta}}(v_{i_1j_{i_1}})}x_k^{a})\\
 =&(x_{\alpha}^{\d_{{\alpha}}(v_{i_1j_{i_1}})+2-b'},x_k^{a-(d_{i_1k}-1)})x_{\alpha}^{b'}x_{\beta}^{
\d_{{\beta}}(v_{i_1j_{i_1}})}x_k^{d_{i_1k}-1}.
\end{align*}
Therefore, $\reg(L\cap K)=(b'+
\d_{{\beta}}(v_{i_1j_{i_1}})+d_{i_1k}-1)+(\d_{{\alpha}}(v_{i_1j_{i_1}})+2-b')+(a-d_{i_1k}+1)-1=
\d_{{\beta}}(v_{i_1j_{i_1}})+\d_{{\alpha}}(v_{i_1j_{i_1}})+a+1> d+1$ since $a=d-\d_{{\alpha}}(v)-\d_{{\beta}}(v)>d-\d_{{\alpha}}(v_{i_1j_{i_1}})-\d_{{\beta}}(v_{i_1j_{i_1}})$.

Since both $K^{\calP}$ and $J^{\calP}$ are the edge ideals of induced subhypergraphs  of $H$,  $\reg(K)=\reg(J)=d$ by Lemmas \ref{polar} and \ref{hyper}.
By Lemma $\ref{exact}(2)$ and the  exact sequence
	\begin{equation*}
		0\longrightarrow \frac{S}{K \cap L}\longrightarrow \frac{S}{K} \oplus
\frac{S}{L}\longrightarrow \frac{S}{K+L} \longrightarrow 0,
		\label{eqn:SES-4}
	\end{equation*}	
we can get $\reg(K+L)\geq d+1$,  a contradiction.
\end{proof}

\begin{Lemma}\label{e6}
Let $I\subset S=\KK[x_1,x_2,x_3]$ be an equigenerated monomial ideal  as in Setting \ref{setting1}   satisfying condition $(**)$ of Definition \ref{condition2}. Then $\reg(I)=d$.
\end{Lemma}
\begin{proof} Let $I=\sum\limits_{i=1}^{t}I_i$, where each $I_i$ has the form described in Setting \ref{setting1}. By induction on $t$, we will prove that $I$ has linear quotients, which implies that $\reg(I)=d$. Case $t=2$ follows from the proof of Lemma \ref{e4.1}.
Assuming $t\ge 3$ and $\sum\limits_{i=1}^{t-1}I_i$ has linear quotients. Thus there exists an order
 of  minimal generators of $\sum\limits_{i=1}^{t-1}I_i$, say $v_{11},\ldots,v_{1\ell_1},\ldots, v_{(t-1)1},\ldots,v_{(t-1)\ell_{(t-1)}}$, such that it has  linear quotients.
Let $v_{ij}=u_{ij}x_3^{c_i}$ for each $i\in[t]$ and $j\in [\ell_i]$. By  symmetry and the condition $(2^{**})$ in Definition \ref{condition2}, we  can assume that
 $\d_{1}(u_{tj_1})=\d_{1}(u_{(t-1)k_1})$ for some $j_1\in [\ell_t]$ and $k_1\in [\ell_{t-1}]$.
Thus  $((v_{(t-1)k_1}):v_{tj_1})=((u_{(t-1)k_1}x_3^{c_{t-1}}):u_{tj_1}x_3^{c_{t}})=(x_2)$ since $c_{t-1}=c_t-1$. Furthermore, $\d_{2}(v_{(t-1)k_1})=d-\d_{1}(v_{(t-1)k_1})-c_{t-1}=d-\d_{1}(v_{tj_1})-(c_{t}-1)=\d_{2}(v_{tj_1})+1$.

 We first show that the ideal  $((v_{11},\ldots,v_{1\ell_1},\ldots, v_{(t-1)1},\ldots,v_{(t-1)\ell_{t-1}}):v_{tj_1})$, i.e. $((\sum\limits_{i=1}^{t-1}I_i):v_{tj_1})$,  is generated by a subset of the variables. We divide into  the following two cases:

(i) If $\d_{2}(v)\geq\d_{2}(v_{tj_1})+1$ for any $v\in\mathcal{G}(\sum\limits_{i=1}^{t-1}I_i)$, then $((\sum\limits_{i=1}^{t-1}I_i):v_{tj_1})=(x_2)$, since $(v_{(t-1)k_1}):v_{tj_1}=(x_2)$.

(ii) If $\d_{2}(v)\le\d_{2}(v_{tj_1})$ for some $v\in\mathcal{G}(\sum\limits_{i=1}^{t-1}I_i)$, note that $\d_{2}(v_{(t-1)k_1})=\d_{2}(v_{tj_1})+1$. Then there exists some $w\in\mathcal{G}(\sum\limits_{i=1}^{t-1}I_i)$ such that $\d_{2}(w)=\d_{2}(v_{tj_1})$ by Lemma \ref{e5} and the condition $(3^{**})$ in Definition \ref{condition2} since $\sum\limits_{i=1}^{t-1}I_i$ has linear quotients. We can choose $w$ to be a generator in $\mathcal{G}(\sum\limits_{i=1}^{t-1}I_i)$  with the smallest degree over $x_1$ such that $\d_{2}(w)=\d_{2}(v_{tj_1})$, then $\d_{1}(w)=\d_{1}(v_{tj_1})+1$. Indeed, since
$\d_{3}(v_{tj_1})>\d_{3}(w)$ and $\deg(w)=\deg(v_{tj_1})$, $\d_{1}(w)>\d_{1}(v_{tj_1})$. If $\d_{1}(w)- \d_{1}(v_{tj_1})\geq 2$,
then by the condition $(4^{**})$ in Definition \ref{condition2}, we see that  $\d_{1}(v)>\d_{1}(v_{tj_1})$ or $\d_{3}(v)\ge\d_{3}(v_{tj_1})$  for any $v\in\mathcal{G}(\sum\limits_{i=1}^{t-1}I_i)$.
However, for the monomial $v_{(t-1)k_1}$ in $\mathcal{G}(\sum\limits_{i=1}^{t-1}I_i)$, we already know that $\d_{1}(v_{(t-1)k_1})=\d_{1}(v_{tj_1})$ and $\d_{3}(v_{(t-1)k_1})=\d_{3}(v_{tj_1})-1$. It is impossible.
In  other words, there exists some $w\in\mathcal{G}(\sum\limits_{i=1}^{t-1}I_i)$ such that $\d_{2}(w)=\d_{2}(v_{tj_1})$ and $\d_{1}(w)=\d_{1}(v_{tj_1})+1$.
Thus $((\sum\limits_{i=1}^{t-1}I_i):v_{tj_1})=(x_1,x_2)$.

Next, we will show that $I$ has  linear quotients.  We will present an order  $v_{11},\ldots,v_{1\ell_1},\\
\ldots, v_{(t-1)1},\ldots,v_{(t-1)\ell_{(t-1)}},v_{t1}, \ldots, v_{t\ell_t}$
 of  minimal generators in  $I$, where  $v_{t1}, \ldots, v_{t\ell_t}$ are taken similarly to  $v_{21}, \ldots, v_{2\ell_2}$
 in Lemma \ref{e4.1}. If  $\ell_t=1$, then  $I$ has linear quotients by the (i) and (ii) above. If  $\ell_t\ge 2$, then we  replace $v_{11},\ldots,v_{1\ell_1},\ldots, v_{(t-1)1},\ldots,v_{(t-1)\ell_{t-1}}$, $\sum\limits_{i=1}^{t-1}I_i$ and  $v_{t1}, \ldots, v_{t\ell_t}$ by $v_{11},\ldots,v_{1\ell_1}$, $I_1$ and    $v_{21}, \ldots, v_{2\ell_2}$, respectively. By similar arguments as in the proof of Lemma \ref{e4.1}, we can deduce that $I$ has linear quotients.
We omit its proof.
\end{proof}

From the above series of lemmas, we obtain

\begin{Theorem}\label{e7}
Let $I\subset S=\KK[x_1,x_2,x_3]$ be an equigenerated monomial ideal of degree $d$.  Then the following statements are equivalent:

$(1)$ $\reg(I)=d$;

$(2)$ $I$ has a linear resolution;

$(3)$ $I$ has linear quotients.
\end{Theorem}
\begin{proof}
The equivalence of $(1)$ and $(2)$ is clear. $(3)$ implies $(2)$ follows from \cite[Lemma 4.1]{CH}. $(1)\Rightarrow(3)$:  Let  $I=\sum\limits_{i=1}^{t}I_i$ be an equigenerated monomial ideal as in Setting \ref{setting1}. The case for $t=1$ follows from Lemma \ref{e1}, and   the case for $t\ge 2$ follows from the proofs of Lemma \ref{e6}.
\end{proof}

Finally, we will demonstrate that  Conjecture \ref{conj_KP} holds for equigenerated monomial
ideals when $n=3$.
\begin{Theorem}\label{e8}
Let $I\subset S=\KK[x_1,x_2,x_3]$ be an equigenerated monomial ideal  of degree $d$ as in Remak \ref{3}, where  $s\ge 2$. If $\reg(I)=d$, then $\reg(\overline{I})=d$.
\end{Theorem}
\begin{proof}
  Let $I=\sum\limits_{i=1}^{s}J_i$, where $J_i=(v_{i1},\ldots, v_{ip_i})x_k^{d_{ik}}$ with $d_{1k}<\cdots<d_{sk}$, and  each  $v_{ij}=x_{\alpha}^{e_{ij}}x_{\beta}^{f_{ij}}\in \KK[x_{\alpha},x_{\beta}]$, where  $\alpha, \beta\in [3]\setminus \{k\}$ are distinct and satisfy $0\le e_{i1}<e_{i2}<\cdots<e_{ip_i}$ and $f_{i1}>f_{i2}>\cdots>f_{ip_i}\ge 0$ for any $i\in [s]$.
 Also, let $\overline{I}=\sum\limits_{i=1}^{s'}J'_i$, where $J'_i=(w_{i1},\ldots, w_{iq_i})x_k^{d'_{ik}}$ with $d'_{1k}<\cdots<d'_{s'k}$, and  each  $w_{ij}=x_{\alpha}^{e'_{ij}}x_{\beta}^{f'_{ij}}\in \KK[x_{\alpha},x_{\beta}]$, where  $\alpha, \beta\in [3]\setminus \{k\}$ are distinct and satisfy $0\le e'_{i1}<e'_{i2}<\cdots<e'_{iq_i}$ and $f'_{i1}>f'_{i2}>\cdots>f'_{iq_i}\ge 0$ for any $i\in [s']$.

We first show that $s=s'$. Let $\mathcal{G}(I)=\{\xb^{{\bb}_1},\ldots, \xb^{{\bb}_m}\}$, then by Lemma \ref{e0}, $\overline{I}$ is generated by the
monomials  $\xb^{\ab}$, where $\ab=(\lceil a \rceil,\lceil b \rceil,\lceil c \rceil)$ with $(a,b,c)\in \text{conv}(\mathcal{G}(I))$. By the definition of $\text{conv}(\mathcal{G}(I))$,
we know that  $(a,b,c)=\sum\limits_{i=1}\limits^{m}{\lambda_i{\bb}_i}$, where $\sum\limits_{i=1}\limits^{m}{\lambda_i}=1$ and each  $\lambda_i \in \mathbb{Q}_{+}$.
For every   $x_k$, by comparing its degree   in $\xb^{\ab}$, we get that
 \begin{align*}
 \d_{k}(\xb^{\ab})&=\lceil\sum\limits_{i=1}\limits^{m}({\lambda_i\d_{k}({\bb}_i}))\rceil\ge\sum\limits_{i=1}\limits^{m}(\lambda_i d_{1k})=d_{1k}, \\ \d_{k}(\xb^{\ab})&=\lceil\sum\limits_{i=1}\limits^{m}({\lambda_i\d_{k}({\bb}_i}))\rceil\le\lceil\sum\limits_{i=1}\limits^{m}(\lambda_i d_{sk})\rceil=d_{sk}.
  \end{align*}
From the expression of $I=\sum\limits_{i=1}^{s}J_i$, we deduce that $s=s'$.

 Then it suffices to prove that $\overline{I}$ is an equigenerated monomial ideal of degree $d$. In this condition, we claim: $\overline{I}$ satisfies the conditions $(**)$. By the claim and Lemma \ref{e6}, we have $\reg(\overline{I})=d$.

 proof of the claim: By the expressions of $I$ and $\overline{I}$ and the fact that $\mathcal{G}(I)\subseteq \mathcal{G}(\overline{I})$ and $s=s'$, we have $\mathcal{G}(J_i)\subseteq \mathcal{G}(J'_{i})$ for every  $i\in [s]$.
 Thus the conditions $(2^{**})$ and $(3^{**})$ hold in Definition \ref{condition2}. For every $i\in [s]$, since  $\overline{I}$ is  integrally closed,  we have $e'_{i(j+1)}-e'_{ij}=1$ and $f'_{ij}-f'_{i(j+1)}=1$
 for each $j\in [p_i]$  by Lemma \ref{e1}. So the condition $(4^{**})$ in Definition \ref{condition2} not occur.

Under the condition that $\overline{I}$ is an equigenerated monomial ideal, we will show the condition $(1^{**})$ holds by induction on $s$. The case $s=2$ follows from  $J'_1=\overline{J_1}=J_1$, which implies that $\reg(J'_1)=d$. In the following, we assume that $s\ge 3$ and that $\reg(\sum\limits_{i=1}^{\ell}J'_i)=d$ for  every  $\ell\in [s-2]$.
From the proof above, $\sum\limits_{i=1}^{\ell}J'_i$ satisfies the conditions  $(2^{**})$ and $(3^{**})$ in Definition \ref{condition2}. Furthermore,  the condition $(4^{**})$ in Definition \ref{condition2} not occur. On the other hand, according to  Lemma \ref{e5}, since $\reg(I)=d$,   we get $d_{(i+1)k}-d_{ik}=1$ for any $i\in [s-1]$.
We can deduce that $\reg(\sum\limits_{i=1}^{s-1}J'_i)=d$ by Lemma \ref{e6} and  induction.

 Finally, we will prove that $\overline{I}$ is an equigenerated monomial ideal. For any monomial  $f\in \mathcal{G}(\overline{I})\setminus \mathcal{G}(I)$,  by Lemma \ref{e0}, we can get $f=x_1^{\lceil a \rceil}x_2^{\lceil b \rceil}x_3^{\lceil c \rceil}$ with $(a,b,c)\in \text{conv}(\mathcal{G}(I))$. If $c\in \mathbb{Q}\setminus \mathbb{N}$, then $(a,b,c)$ can be seen as a point on a line segment with $\ab=(m_1,n_1,\lceil c \rceil)$ and $\bb=(m_2,n_2,\lfloor c \rfloor)\in \text{conv}(\mathcal{G}(I))$ as endpoints. By symmetry, assume that $m_1<m_2$. In fact, if there does note exist $m_1$, $m_2$ such that $m_1\neq m_2$, in this case we have $m_1=m_2\in \mathbb{N}$, then we have $x_1^{m_1}x_2^{n_1}x_3^{\lceil c \rceil}\mid f$. Since $(a,b,c)=\sum\limits_{i=1}\limits^{m}{\lambda_i{\bb}_i}$ with $\sum\limits_{i=1}\limits^{m}{\lambda_i}=1$, we have $a+b+c=d$, which implies  that $\lceil a \rceil+ \lceil b \rceil+\lceil c \rceil>d$.
This contradicts  the hypothesis that  $f\in \mathcal{G}(\overline{I})$.  Thus  $(a,b,c)=\lambda\ab+(1-\lambda)\bb$, where $\lambda\in \mathbb{Q}_{+}$ such that $0<\lambda<1$.
By comparing  the components  of the expression  $(a,b,c)=\lambda\ab+(1-\lambda)\bb$, we get  $m_1<a<m_2$.
Since $\text{conv}(\mathcal{G}(I))$ is the convex hull of $\mathcal{G}(I)$, we have  $(\lceil m_1 \rceil, \lfloor n_1 \rfloor, \lceil c \rceil)\in\text{conv}(\mathcal{G}(I))$ or $(\lfloor m_1 \rfloor,\lceil n_1 \rceil,\lceil c \rceil)\in\text{conv}(\mathcal{G}(I))$.
Under the condition that  $(\lceil m_1 \rceil, \lfloor n_1 \rfloor, \lceil c \rceil)\in\text{conv}(\mathcal{G}(I))$, there are two subcases
(I) If $n_1\le b$, then  $x_1^{\lceil m_1 \rceil}x_2^{\lfloor n_1 \rfloor}x_3^{\lceil c \rceil}\mid f$, a contradiction.
 (II) If $n_1>b$, then $n_1>b>n_2$. Consider the sequences of points $\ab_i=(a_i,b_i,\lceil c \rceil)$, $\bb_j=(a'_j,b'_j,\lfloor c \rfloor)\in\text{conv}(\mathcal{G}(I))$ where $a_i=\lceil m_1 \rceil+i-1$, $b_i=\lfloor n_1 \rfloor-i+1$, $a'_j=\lfloor m_2 \rfloor-i+1$ and $b'_j=\lceil n_2 \rceil+i-1$. By the condition $(2^{**})$ in Definition \ref{condition2}, we can deduce that $a_i=a'_j$  or $b_{i}=b_{j'}$ for some  $i,j$.
It follows that $x_1^{a_{i'}}x_2^{b_{i'}}x_3^{\lceil c \rceil}\mid f$ for some $i'$  or $x_1^{a'_{j'}}x_2^{b'_{j'}}x_3^{\lfloor c \rfloor}\mid f$ for some $j'$. This implies that  $f\notin \mathcal{G}(\overline{I})$, a contradiction.
By similar arguments as the condition that  $(\lceil m_1 \rceil, \lfloor n_1 \rfloor, \lceil c \rceil)\in\C(I)$,  if $(\lfloor m_1 \rfloor,\lceil n_1 \rceil,\lceil c \rceil)\in\C(I)$,
we also  get the contradiction. Therefore, $c\in  \mathbb{N}$. Likewise, $a,b\in  \mathbb{N}$. So $\overline{I}$ is generated by monomials  of degree $d$.
We complete the proof.
\end{proof}

Before concluding this paper, we provide the proof of Theorem \ref{e9}.

\medskip

\textit{ The proof of \textit{  of Theorem \ref{e9}}}\,:
 If $I=(u_{1},\ldots, u_{\ell})x_3^{c_1}$, then $\reg(\overline{I})\leq \reg(I)$ holds by Theorem \ref{2-variable}.
Assume that $I=\sum\limits_{i=1}^{t}I_i$, where $t\ge 2$ and  each $I_i$ has the form described in Setting \ref{setting1}.
If $\reg(I)=d$, then $\reg(\overline{I})=\reg(I)$ by Theorem \ref{e8}.
If $\reg(I)\ge d+1$, then  the primary decomposition of  $I$ allows us to  deduce that $\dim(S/I)=0$ or that $I=X^{\mathbf{a}}J$, where $J$ is  a  monomial ideal such that
 $\dim(S/J)\le 1$, or $J=S$. If $I=X^{\mathbf{a}}J$, then $\reg(I)=\reg(J)+\deg(X^{\mathbf{a}})$ and
 $\overline{I}=X^{\mathbf{a}}\overline{J}$, meaning that $\reg(\overline{I})=\reg(\overline{J})+\deg(X^{\mathbf{a}})$.
So it suffices to consider the case when $\dim(S/I)\leq 1$. By Lemma \ref{compare}, $\reg(\overline{I})\leq\delta(I)+1=d+1\leq\reg(I)$, as desired.

\medskip

\medskip
\hspace{-6mm} {\bf Acknowledgments}

 \vspace{3mm}

\hspace{-6mm} The authors would like to sincerely express their gratitude to Professor Tang Zhongming for his participation in our discussions and his valuable suggestions,  as well as to Professor Le Tuan Hoa for his insightful contributions.
This research is supported by the Natural Science Foundation of Jiangsu Province
(No. BK20221353) and the National Natural Science Foundation of China (No. 12471246).

	\end{document}